\documentclass[]{interact}

\usepackage{epstopdf}
\usepackage[caption=false]{subfig}

\usepackage[numbers,sort&compress]{natbib}
\bibpunct[, ]{[}{]}{,}{n}{,}{,}
\renewcommand\bibfont{\fontsize{10}{12}\selectfont}

\theoremstyle{plain}
\newtheorem{theorem}{Theorem}[section]
\newtheorem{lemma}[theorem]{Lemma}
\newtheorem{corollary}[theorem]{Corollary}
\newtheorem{proposition}[theorem]{Proposition}
\newtheorem{property}[theorem]{Property}

\theoremstyle{definition}

\theoremstyle{remark}

\usepackage{amsmath,amssymb,amsthm}
\usepackage{enumitem}
\usepackage{geometry}
\usepackage[mathscr]{eucal}

\newcommand{\iti}[1]{{\em(#1)}}
\newcommand{\Span}{\operatorname{span}}

\def\n0{_{n\geq0}}
\def\up#1{^{(#1)}}
\def\upn{\up{n}}
\def\upnp{\up{n+1}}
\def\upnm{\up{n-1}}

\def\'{\hspace{1pt}}
\def\dom{{\cal D}}
\def\calB{{\mathcal B}}

\def\calH{{\cal H}}

\def\calK{{\mathcal K}}
\def\calL{{\mathcal L}}

\def\calS{{\mathcal S}}
\def\nul{{\cal N}}
\def\ran{{\cal R}}

\DeclareMathOperator*{\argmin}{argmin}
\def\smallsum#1{{\textstyle\sum\limits_{#1}}}
\newcommand{\midfrac}[2]{\mbox{\small$\displaystyle\frac{#1}{#2}$}}
\newcommand{\middfrac}[2]{\mbox{\footnotesize$\displaystyle\frac{#1}{#2}$}}
\def\NN{{\mathbb N}}
\def\RR{{\mathbb R}}
\def\eps{\epsilon}
\def\lin{\mathrm{lin}}
\def\f{{\mathrm f}}

\def\p{{\mathrm p}}
\def\r{{\mathrm r}}

\def\ppnoi{\par\medskip\noindent}
\def\line{\par\noindent}
\def\page{\vfill\eject}

\def\scC{{\mathscr C}}

\def\scR{{\mathscr R}}
\def\scS    {{\mathscr S}}

\def\where{\qquad\text{where}\qquad}
\def\et{\qquad\text{and}\qquad}

\begin{document}

\title{A residual-iteration framework for alternating projections between affine subspaces}

\author{
\name{Nguyen T. Thao}\thanks{CONTACT Nguyen T. Thao. Email: tnguyen@ccny.cuny.edu}
\affil{Dept. of Electrical Engineering, City College of New York, CUNY, New York, USA}
}

\maketitle

\begin{abstract}
We reformulate the problem of alternating projections between two affine subspaces of a Hilbert space as the minimization of a least-squares functional associated with a bounded linear operator. This viewpoint reveals that classical alternating projections coincide with the unit-step Landweber iteration and enables the introduction of a general residual-state iteration framework that encompasses Landweber, its steepest-descent variant, and the conjugate-gradient method. Within this framework, we establish abstract convergence principles based on residual extinction and translation equivariance, allowing convergence analyses to be carried out once at the level of least-squares optimization and then transferred directly to alternating projection algorithms. As applications, we obtain new variants of alternating projections accelerated by steepest descent and conjugate gradients, together with convergence guarantees in both the consistent and inconsistent settings. We also establish linear convergence results under closed-range assumptions and express the convergence rates explicitly in terms of the Friedrichs angle and the largest principal angle between the underlying subspaces.
\end{abstract}


\section{Introduction}

In a Hilbert space $(H,\|\cdot\|)$, let $U$ and $W$ be two closed affine subspaces. A famous algorithm to find an element in $U\cap W$ when it exists consists of alternating orthogonal projections between $U$ and $W$.
More precisely, the iteration is
\begin{equation}\label{AP}
u\upnp=P_U P_W u\upn,\qquad n\geq0
\end{equation}
where $P_S$ denotes the orthogonal projection onto a closed affine subspace $S$.  When $U\cap W$ is empty, this iteration is also known to converge to an element of $U$ that is closest to $W$ when it exists \cite{Bauschke94}. Defining the set
\begin{equation}\label{UW}
M_{U,W}:=\Big\{u\in U:\|u\!-P_W u\|=\displaystyle\inf_{v\in U}\|v-P_W v\|\Big\},
\end{equation}
it is specifically shown in Theorem 4.1 of \cite{Bauschke94} that $(u\upn)\n0$ satisfies the following property.

\begin{property}\label{prop:Bauschke}~
\begin{enumerate}[label=(\roman*)]
\setlength{\itemsep}{0.2em}
\item If $M_{U,W}\neq\emptyset$, $u\upn$  converges to $P_{M_{U,W}}u\up0$.
\item If $M_{U,W}=\emptyset$, $\|u\upn\|$ tends to $\infty$.
\end{enumerate}
\end{property}
\ppnoi
In this paper, all limits are implicitly in the sense of convergence in norm, also called strong convergence. In the trend of algorithmic acceleration, the main purpose of this article is to seek alternative methods that yields Property \ref{prop:Bauschke}.
Without loss of generality, we will assume that $U$ contains 0, which makes it a linear subspace. We will also assume that the initial iterate $u\up0$ is in $U$.

The starting point of this paper is the observation that iteration \eqref{AP} coincides exactly with the gradient-descent iteration of step size 1 for the functional
\begin{equation}\label{Phi0}
\Phi_{U,W}(u):=\tfrac12\,\big\|u-P_W u\big\|^2,\qquad u\in U.
\end{equation}
Now, one can always write that
\begin{equation}\label{PW-gen}
u-P_W u=Qu-w,\qquad u\in U
\end{equation}
by taking
\begin{equation}\label{Qw-AP}
Qu:=u-(P_W u-P_W\'0),\quad u\in U \et w:=P_W\'0.
\end{equation}
Since $P_W$ is a non-expansive affine transformation, then $Q$ is a bounded linear operator of domain $U$.
Then, $\Phi_{U,W}(u)$ takes the form of
\begin{equation}\label{PhiQw}
\Phi_{Q,w}(u):=\tfrac12\,\big\|Qu-w\big\|^2,\qquad u\in U.
\end{equation}
Minimizing this quadratic functional falls within the classical framework of linear inverse problems \cite{Engl96}. There are several standard iterative algorithms that perform this type of minimization.  The first part of the paper is to extract from these algorithms a family of methods that have the potential to achieve Property \ref{prop:Bauschke} once $Q$ and $w$ are chosen as in \eqref{Qw-AP}. In this investigation, we will assume that $Q$ is any bounded linear operator from $U$ to $\hat U$.\footnote
{While we are assuming that $U$ and $\hat U$ are independent Hilbert spaces, we will use the same symbol for the norms in  $\hat U$ and in $U$ for  notation simplicity. Any ambiguity is resolved by the domain of the argument.} More precisely, we proceed as follows. Defining the set of minimizers
\begin{equation}\label{MQw}
M_{Q,w}:=\displaystyle\Big\{u\in U:\|Qu-w\|=\inf_{v\in U}\|Qv-w\|\Big\},
\end{equation}
we will seek algorithms with iterates $(u\upn)\n0$ in $U$ that satisfy the following property.

\begin{property}\label{obj}
For any $w\in\hat U$ and $u\up0\in U$,
\begin{enumerate}[label=(\roman*)]
\setlength{\itemsep}{0.2em}
\item if $M_{Q,w}\neq\emptyset$, $u\upn$  converges to $P_{M_{Q,w}}u\up0$,
\item if $M_{Q,w}=\emptyset$, $\|u\upn\|$ tends to $\infty$.
\end{enumerate}
\end{property}
\line
Property \iti{ii} is not a desirable objective in itself. Rather, it characterizes the intrinsic behavior of the proposed methods when no least-squares minimizer exists. Once such a method has been obtained, taking $Q$ and $w$ from \eqref{Qw-AP} automatically yields a corresponding alternating projection algorithm satisfying Property \ref{prop:Bauschke}.

The paper is organized as follows. Section 2 establishes the least-squares formulation of alternating projections and shows that the classical alternating projection method is exactly the unit-step Landweber iteration. Section 3 introduces a general framework of admissible residual-state iterations, within which residual extinction and translation equivariance provide abstract tools for deriving strong convergence results independently of any particular algorithm. Sections 4, 6 and 7 then show that the Landweber, steepest-descent and conjugate-gradient methods all fit this framework, leading to corresponding variants of alternating projections together with convergence guarantees in both the consistent and inconsistent settings. Section 5 specializes these results to alternating projections and relates the convergence behavior to the principal angles between the underlying subspaces.

In addition to strong convergence, we investigate linear convergence under suitable closed-range assumptions. Besides providing quantitative convergence guarantees, these estimates make it possible to compare analytically the convergence speeds of the different methods in the infinite-dimensional setting, where numerical experimentation alone cannot capture the general behavior of the algorithms. In particular, the resulting convergence rates provide a theoretical comparison between classical alternating projections and their steepest-descent and conjugate-gradient accelerations.

\section{From alternating projections to Landweber iteration}\label{sec:AP-LI}

In this section, we show that \eqref{AP} coincides with a gradient descent iteration applied to the functional $\Phi_{U,W}$ in \eqref{Phi0}. This observation serves as the foundation for the algorithmic generalizations developed in the sequel.

\subsection{Least-squares representation of projection residual}

Given the spaces $U$ and $W$ of the introduction, the first step is to formalize the operator $Q$ in \eqref{Qw-AP} which yields the relation \eqref{PW-gen}.
Let $V$ denote the direction subspace of $W$. With $w=P_W\'0$ as defined in \eqref{Qw-AP}, we have
\begin{equation}\label{WVw}
W=V+w
\end{equation}
since $w\in W$. Consequently, $w$ must be in $V^\perp$ as well. To express $P_W$, it is convenient to recall the following elementary result from Lemma 2.4 of  \cite{Bauschke24}.
\line
\begin{lemma}\label{lem:Bauschke}
Let $A$ be a closed affine subspace of $H$ and $s\'\in H$. Then,
\begin{equation}\label{P-shift}
\forall u\in H,\qquad P_{ A+s}\'u=P_{ A}\'(u-s)+s.
\end{equation}
\end{lemma}
\ppnoi
It follows from \eqref{WVw} that
\begin{equation}\label{PW}
P_W u=P_V(u-w)+w=P_V u+w,\qquad u\in H
\end{equation}
since $w\in V^\perp$. Then \eqref{Qw-AP} implies that $Qu=u-P_V u=P_{V^\perp}u$ for all $u\in U$. Thus, \eqref{PW-gen} is obtained by taking
\begin{equation}\label{Q}
\begin{array}[t]{rcl}
Q: \quad U& \rightarrow &V^\perp\\
 u& \mapsto &P_{V^\perp}u
\end{array}
\et w=P_W\'0\in V^\perp.
\end{equation}
Since the function $\Phi_{U,W}(u)$ defined in \eqref{Phi0} is equal to $\Phi_{Q,w}(u)$ with the above definition of $Q$ and $w$, minimizing \eqref{PhiQw} is the more general problem that we shall study. We now study this problem.

\subsection{Gradient descent for a general least-squares functional}\label{subsub:gradient-descent}

Assume now that $Q\in\calB(U,\hat U)$, the space of bounded linear operators between Hilbert spaces $(U,\|\cdot\|)$ and $(\hat U,\|\cdot\|)$
and let $w$ be a given element of $\hat U$. A basic method for minimizing the functional $\Phi_{Q,w}$ in \eqref{PhiQw} denotes the gradient descent iteration
\begin{equation}\label{gradient-descent}
u\upnp=u\upn-\alpha\,\nabla\Phi_{Q,w}(u\upn),\qquad n\geq0
\end{equation}
where $\nabla\Phi(u)$ is the gradient of $\Phi$ at $u$ and $\alpha$ is a positive constant chosen to ensure convergence. Explicitly \cite[p.4]{Kaltenbacher08},
\begin{equation}\label{residual}
\nabla\Phi_{Q,w}(u)=Q^*Qu-Q^*w=Q^*(Qu-w),\qquad n\geq0
\end{equation}
where $Q^*$ is the adjoint of $Q$. Hence, the complete gradient descent operation is
\begin{equation}\label{Landweber}
u\upnp= u\upn-\alpha\,Q^*(Qu\upn\!-w),\qquad n\geq0.
\end{equation}
This iteration is known as the {\em Landweber iteration}.
The following result is adapted from Theorem 6.1 of \cite{Engl96}.

\begin{theorem}[\cite{Engl96}]\label{theo0}
Let $Q$ be a non-expansive linear operator and $(u\upn)_{n\geq0}$ be generated by the Landweber iteration in \eqref{Landweber} with $\alpha=1$ starting from $u\up0=0$.
\begin{enumerate}[label=(\roman*)]
\setlength{\itemsep}{0.2em}
\item  If $M_{Q,w}\neq\emptyset$, $u\upn$ converges to\'\footnote
{\label{foot:Q+w}The limit is presented in \cite{Engl96} to be $Q^\dagger w$, which is the minimum-norm element of $M_{Q,w}$ as will be seen in Section \ref{subsec:least-squares}. This coincides with $P_{M_{Q,w}}0$.} $P_{M_{Q,w}}\'0$.
\item If $M_{Q,w}=\emptyset$, $\|u\upn\|$ tends to $\infty$.
\end{enumerate}
\end{theorem}
\ppnoi
This theorem establishes Property \ref{obj} in the case where $u\up0=0$.

\subsection{Alternating projections as unit-step Landweber iteration}

We now return to the original alternating projection problem by taking the operator $Q$ of \eqref{Q}.
In this case, $\Phi_{Q,w}=\Phi_{U,W}$ and hence
\begin{equation}\label{UU}
M_{Q,w}=M_{U,W}
\end{equation}
where $M_{U,W}$ was defined in \eqref{UW}. One readily verifies that the adjoint $Q^*$ of $Q$ is defined by
\begin{equation}\label{Q*}
\begin{array}[t]{rcl}
Q^*: \quad V^\perp & \rightarrow &U\\
v& \mapsto &P_U\'v
\end{array}.
\end{equation}
Indeed, for all $u\in U$ and $ v\in V^\perp$, $\langle P_{V^\perp}u,v\rangle=\langle u,P_{V^\perp}v\rangle=\langle u,v\rangle=\langle P_Uu,v\rangle=\langle u,P_Uv\rangle$.
Equation \eqref{PW-gen} then gives
\begin{equation}\label{AP-res}
Q^*(Qu\upn\!-w)=P_U\big(u\upn \!-P_W u\upn\big)=u\upn\!-P_U P_W u\upn
\end{equation}
since $u\upn\in U$ for all $n\geq0$. Then, the Landweber iteration \eqref{Landweber} becomes
\begin{align}
u\upnp=u\upn+\alpha\,\big(P_U P_W u\upn\!-u\upn\big),\qquad n\geq0.\label{POCS-const}
\end{align}
When $\alpha=1$, we recover the original alternating projections of \eqref{AP}. This shows that \eqref{AP} is a gradient descent iteration for $\Phi_{U,W}$. Meanwhile, Theorem \ref{theo0} yields Theorem \ref{prop:Bauschke} at least when $u\up0=0$, since $Q$ is nonexpansive and hence $\alpha=1\in(0,1/\|Q\|^2]$.

\subsection{General properties of least-squares solutions}\label{subsec:least-squares}

Theorem \ref{theo0} has helped explain a major part of Theorem \ref{prop:Bauschke} from the broader perspective of minimizing a least-squares functional $\Phi_{Q,w}$. In preparation for algorithmic generalizations, we review several important properties on the set of minimizers $M_{Q,w}$ in infinite dimension. We assume again $Q\in\calB(U,\hat U)$. We use the following notation:
\begin{description}
\setlength{\itemsep}{0.2em}
\item[$\dom(Q)$]: domain of $Q$,
\item[$\ran(Q)$]: range of $Q$,
\item[$\nul(Q)$]: null space of $Q$.
\end{description}

\begin{proposition}\label{prop:MQw}
Let $u\in U$. Then, $u\in M_{Q,w}$ if and only if any of the following properties is satisfied:
\begin{enumerate}[label=(\roman*)]
\setlength{\itemsep}{0.4em}
\item $Q u=\bar w$ where $\bar w:=P_{\,\overline{\ran(Q)}}\'w$,\vspace{-1.5mm}
\item $Qu-w\in\ran(Q)^\perp$,
\item $u$ is a solution of the normal equation $Q^*Q u=Q^*w$,
\item $u\in\nul(Q)+u^*$ for some $u^*\in M_{Q,w}$.
\end{enumerate}
\end{proposition}

\begin{proof}
$(i)$: We have\vspace{-2mm}
\begin{equation}\label{w-bw}
 Qu-w=( Qu-\bar w)+(w-\bar w)~~\text{where}~~
 Qu-\bar w\in\overline{\ran(Q)}~~\text{and}~~ w-\bar w\in\ran(Q)^\perp.\vspace{-1mm}
\end{equation}
By the Pythagorean theorem, $\|Qu-w\|^2=\| Qu-\bar w\|^2+\|w-\bar w\|^2$. As $\| Qu-\bar w\|$ can be arbitrarily close to 0,  $\|Qu-w\|$ is minimized if and only if $\| Qu-\bar w\|=0$, i.e., $Qu=\bar w$.

$(ii)$:  It follows from \eqref{w-bw} that $Qu-w\in\ran(Q)^\perp\Leftrightarrow Qu-\bar w\in\ran(Q)^\perp\Leftrightarrow Qu-\bar w=0$.

$(iii)$: This is the well-known characterization of least-squares solutions, see for example  \cite[\S2.1]{Engl96}).

$(iv)$:  This follows using  {\em(iii)}, and the  fact that $\nul(Q^*Q)=\nul(Q)$.
\end{proof}
\ppnoi
Whenever $M_{Q,w}$ is nonempty, Proposition \ref{prop:MQw}\iti{iv} implies that $M_{Q,w}=\nul(Q)+u^*$ with any $u^*\in M_{Q,w}$. In particular
\begin{eqnarray}
&M_{Q,w}=\nul(Q)+Q^\dagger w\label{MQd}\\
\text{where}\qquad\qquad\qquad\qquad\qquad&&\qquad\qquad~\quad\qquad\qquad\qquad\qquad\nonumber\\
&Q^\dagger w:=\displaystyle\argmin_{u\in M_{Q,w}}\|u\|.\label{Q+}
\end{eqnarray}
The vector $Q^\dagger w$ is the minimum-norm least-squares solution to $Qu=w$. The operator $Q^\dagger$ is linear and is known as the Moore-Penrose pseudoinverse of $Q$ \cite{Luenberger69,wang2018generalized}. Its domain is
\begin{equation}\label{domQ+}
\dom(Q^\dagger)=\big\{w\in\hat U:M_{Q,w}\neq\emptyset\big\}=\ran(Q)\oplus\ran(Q)^\perp
\end{equation}
where the second equality results from Proposition \ref{prop:MQw}\iti{ii}.

As the orthogonal projection $P_{M_{Q,w}}$ will play a central role in this article, the following proposition provides a geometric interpretation of it using $Q^\dagger w$.

\begin{proposition}\label{prop:PM}
Assume that $M_{Q,w}\neq\emptyset$. Then
\begin{equation}\label{PM-Qdw}
Q^\dagger w\in\nul(Q)^\perp \et
P_{M_{Q,w}}u\up0=Q^\dagger w+P_{\nul(Q)}u\up0,\quad u\up0\in U.
\end{equation}
\end{proposition}

\begin{proof}
By \eqref{MQd}, $M_{Q,w}$ is affine of direction $\nul(Q)$. As $Q^\dagger w=P_{M_{Q,w}}0$ from \eqref{Q+}, it follows that $Q^\dagger w\in\nul(Q)^\perp$. Applying Lemma \ref{lem:Bauschke} with $A=\nul(Q)$ and $s=Q^\dagger w$, \eqref{MQd} implies that $P_{M_{Q,w}}u\up0=P_{\nul(Q)}(u\up0\!-Q^\dagger w)+Q^\dagger w=P_{\nul(Q)}u\up0+Q^\dagger w$ since $Q^\dagger w\in\nul(Q)^\perp$.
\end{proof}

\subsection{Closed range condition}\label{subsec:closed-range}

\subsubsection{Systematic existence of least-squares minimizers}\label{subsub:closed-range}

We now determine conditions on $Q$ under which $M_{Q,w}$ is nonempty for every $w\in\hat U$.

\begin{proposition}\label{prop:closed-range}
Let $Q\in\calB(U,\hat U)$. Then, $M_{Q,w}\neq\emptyset$ for all $w\in\hat U$ if and only if $\ran(Q)$ is closed.
\end{proposition}

\begin{proof}
$\ran(Q)$ is closed $\Leftrightarrow$ $\ran(Q)=\overline{\ran(Q)}$ $\Leftrightarrow$ $\ran(Q)=(\ran(Q)^\perp)^\perp$ $\Leftrightarrow$ $\ran(Q)\oplus\ran(Q)^\perp=\hat U$  $\Leftrightarrow$ $M_{Q,w}\neq\emptyset$ for all $w\in\hat U$, due to \eqref{domQ+}.
\end{proof}
\ppnoi
As a consequence of the closed graph theorem in Banach spaces, the following theorem, which is given in Theorem 5.2 of \cite{Kato95}, provides a more analytical characterization for $Q$ to have a closed range.

\begin{theorem}\label{theo:Kato}
Let $Q\in\calB(U,\hat U)$. Then,  $\ran(Q)$ is closed if and only if $\gamma(Q)>0$, where
\begin{equation}\label{gamma0}
\gamma(Q):=\inf_{u\in\nul(Q)^\perp,\|u\|=1}\|Qu\|.
\end{equation}
\end{theorem}
\ppnoi
The scalar $\gamma(Q)$ is called the reduced minimum modulus of $Q$.

\subsubsection{Application to alternating projections}\label{subsub:AP-angle}

Returning to the problem of alternating projections between $U$ and $W$,
applying the preceding results to the pair $(Q,w)$ given in \eqref{Q} allows us to characterize the condition under which $M_{U,W}$ can be guaranteed to be nonempty. The main task is to express $\gamma(Q)$ in terms of characteristics of the subspaces $U$ and $W$.

Under the definition of \eqref{Q}, we show that $\gamma(Q)$ can be expressed in terms of the Friedrichs angle $\theta_\f=\theta_\f(U,V)$ between $U$ and $V$, which is the unique angle of $[0,\frac{\pi}{2}]$ satisfying
\begin{eqnarray}
&\cos\theta_\f:=\sup\,\Bigl\{\,|\langle u,v\rangle| : u\in  U_0,\,v\in  V_0,\,   \|u\|=\|v\|=1\Bigr\}\label{tf}\\
\text{where}\qquad\quad&&\qquad\qquad\qquad\nonumber\\
& U_0:=U\cap(U\!\cap\!V)^\perp \et V_0:=V\cap(U\!\cap\!V)^\perp.\label{U0V0}
\end{eqnarray}

\begin{proposition}
\begin{equation}\label{norm-angle}
\gamma(Q)=\sin\theta_\f(U,V).
\end{equation}
\end{proposition}

\begin{proof}
We first need to establish the following equivalent characterization of $\theta_\f(U,V)$,
\begin{equation}\label{tf-equiv}
\cos\theta_\f=\sup_{u\in U_0,\|u\|=1}\|P_V u\|.
\end{equation}
Since $\sup_{v\in V_0,\|v\|=1}|\langle u,v\rangle|=\|P_{V_0}u\|$, \eqref{tf} yields the equivalent expression $\cos\theta_\f=\sup_{u\in U_0,\|u\|=1}\|P_{V_0}u\|$. But for any $u\in U_0$, $u\in(U\!\cap\!V)^\perp$ and $P_V u-u\in V^\perp\subset(U\!\cap\!V)^\perp$, so that $P_V u\in(U\!\cap\!V)^\perp\cap V=V_0$. Hence, $P_{V_0} u=P_V u$. This proves \eqref{tf-equiv}.

Next, as $Q=P_{V^\perp}|_U$ and $\nul(Q)^\perp=U_0$, it follows from \eqref{gamma0} that $\gamma(Q)=\inf_{u\in U_0,\|u\|=1}\|P_{V^\perp} u\| $. Since
$\|P_{V^\perp}u\|^2=\|u\|^2-\|P_V u\|^2$, then
$\gamma(Q)^2={1-\sup_{u\in U_0,\|u\|=1}\|P_{V} u\|^2}
=1-\cos^2\theta_\f(U,V)=\sin^2\theta_\f(U,V)$. This proves the second equality of \eqref{norm-angle}.
\end{proof}
\ppnoi
As $M_{Q,w}=M_{U,W}$ under \eqref{Q}, we conclude the following result.

\begin{proposition}\label{prop:UW-min}
 Let $U$ and $V$ be two closed linear subspaces. Then, $M_{U,W}$ is nonempty for every affine subspace $W$ with direction subspace $V$ if and only if the Friedrichs angle $\theta_\f(U,V)$ between $U$ and $V$ is positive.
\end{proposition}

\section{A residual-state framework for least-squares iterations}\label{sec:admin-family}

For a nonexpansive operator $Q$, we showed that the Landweber iteration in \eqref{Landweber} with $\alpha=1$ achieves Property \ref{obj} at least partially. Our goal is not only to complete proof of  Property \ref{obj} for this iteration, but to explore a larger family of iterations that achieve this property. Throughout this section, we consider a given operator $Q\in\calB(U,\hat U)$.

\subsection{Family of admissible iterations}\label{subsub:family}

As shown in Proposition \ref{prop:MQw}, a vector $u\in U$ is a minimizer of $\Phi_{Q,w}$ if and only if it is a solution of the normal equation $Q^*Qu=Q^*w$. For the minimization of $\Phi_{Q,w}$, it is natural to focus on algorithms that generate iterates $(u\upn)\n0$ whose residuals
$$r\upn:=Q^*Qu\upn-Q^*w=Q^*(Qu\upn\!-w),\qquad n\geq0$$
converge to zero. We propose to achieve this objective using iterations of the form
\begin{subequations}\label{generic}
\begin{align}
r\upn&=Q^*(Qu\upn\!-w),\label{generic-a}\\
(d\upn,s\upnp)&=J_n(r\upn,s\upn),\label{generic-b}\\
u\upnp&=u\upn+d\upn,\qquad n\geq0\label{generic-c}
\end{align}
\end{subequations}
where $J_n$ is a mapping of the form
$$J_n:~\ran(Q^*)\backslash\{0\}\times \scS_n~\longrightarrow~\ran(Q^*)\times \scS_{n+1}$$
and $(\scS_n)\n0$ is a growing family of state spaces with $\scS_0=\{0\}$. The role of the state variable $s\upn\in\scS_n$ is to store certain quantities derived from previous values of $r\upn$. As $\scS_0=\{0\}$, it follows that $s\up0$ is 0 by default. We also imply that the iteration is terminated whenever $r\upn=0$, to remain consistent with the residual domain of $J_n$.
We call any iteration of the form \eqref{generic} an {\em admissible iteration}. Such an iteration is uniquely defined by the sequence of mappings $(J_n)\n0$. Once such a sequence has been fixed, note that $u\upn$ and $r\upn$ are uniquely determined by $w\in\hat U$ and $u\up0\in U$. We write that
\begin{subequations}\label{URn-def}
\begin{align}
u\upn&=U_n(w,u\up0),\label{Un}\\
r\upn&=R_n(w,u\up0),\qquad n\geq0,~(w,u\up0)\in\hat U{\times}U.\label{Rn}
\end{align}
\end{subequations}
We call $U_n$ and $R_n$ the $n$th iterate map and residual map of the iteration.
Note that
\begin{equation}\label{URn-rel}
R_n(w,u\up0)=Q^*(Q\'U_n(w,u\up0)-w).
\end{equation}
In the case where $M_{Q,w}\neq\emptyset$, we also define
\begin{equation}\label{En}
E_n(w,u\up0):=U_n(w,u\up0)-P_{M_{Q,w}}u\up0,\qquad n\geq0.
\end{equation}
We call $E_n$ the $n$th error map of the iteration, or the error map associated with $U_n$.

\subsection{Standard methods as admissible iterations}\label{subsec:standard}

We now describe the three methods of least-squares quadratic minimization that will be studied in this paper and that fit this framework.

\subsubsection{Landweber iteration with varying step size}\label{subsub:ex-Land-var}

We now study in detail the following general version of Landweber iteration
\begin{equation}\label{Landweber-var}
u\upnp= u\upn-\alpha_n\'Q^*(Qu\upn\!-w),\qquad n\geq0
\end{equation}
where $(\alpha_n)\n0$ is a given varying sequence of coefficients. This can be written as the two-term recurrence
\begin{subequations}\label{Land-res-var}
\begin{align}
r\upn&=Q^*(Qu\upn\!-w),\label{Land-res-var-a}\\
u\upnp&= u\upn-\alpha_n\'r\upn,\qquad n\geq0.\label{Land-res-var-b}
\end{align}
\end{subequations}
This fits the structure of \eqref{generic} with
$$\scS_n=\{0\} \et J_n(r,0)=(-\alpha_n\'r,0),\quad r\in\ran(Q^*),\qquad n\geq0.$$
From a computational perspective, note in this case that $J_n$ does not depend on $Q$.

\subsubsection{Steepest-descent variant of Landweber iteration}\label{subsub:ex-SD}

This is a variant of the Landweber iteration where $\alpha_n$ is chosen adaptively at each iteration. Specifically, $\alpha_n$ is chosen so as to minimize $\Phi_{Q,w}(u\upnp)$. Given that $r\upn=\nabla\Phi_{Q,w}(u\upn)$ from \eqref{residual}, one obtains
\cite[(3.25)]{Nocedal06} that
\begin{equation}\label{an}
\alpha_n:=\argmin_{\alpha\in\RR}\Phi_{Q,w}\big(u\upn\!-\alpha r\upn\big)
=\dfrac{\|r\upn\|^2}{\|Qr\upn\|^2}
\end{equation}
This fits the framework of \eqref{generic} with
$$\scS_n=\{0\}\quad\text{and}\quad J_n(r,0):={\big({-}(\|r\|^2/\|Qr\|^2)\'r,0\big)},\quad r\in\ran(Q^*)\backslash\{0\},\qquad n\geq0.$$
In this case, $J_n$ does not depend on $n$. Note that $\ran(Q^*)\subset\nul(Q)^\perp$ so that $\nul(Q)\cap\ran(Q^*)\backslash\{0\}=\emptyset$.

\subsubsection{Conjugate-gradient method}\label{subsub:CG}

The conjugate-gradient method is a well-known method of least-squares minimization that generally converges faster than the steepest-descent method \cite{Nocedal06}. It is given by the recurrence
\begin{subequations}\label{sys2}
\begin{align}
u\upnp&= u\upn+\alpha_n\'p\upn,\label{sys2a}\\
r\upnp&=r\upn+\alpha_n\'Q^*Q\'p\upn\label{sys2b}\\
p\upnp&=-r\upnp+\beta_n\'p\upn,\quad\qquad n\geq0\qquad\qquad\label{sys2c}\\
\hspace{-2mm}\text{where}\qquad\qquad\qquad\qquad\qquad\qquad~~\,\nonumber\\
\alpha_n:=\|r\upn\|^2\big/\|Qp\upn\|^2& \et
\beta_n:=\|r\upnp\|^2\big/\|r\upn\|^2\label{ab}\\[0.6ex]
\hspace{-8mm}\text{with the initial conditions}\qquad\qquad\nonumber\\
-p\up0=&~r\up0=Q^*(Q u\up0\!-w).\label{sys2d}
\end{align}
\end{subequations}
To verify that this is an admissible iteration, the first step is to see that \eqref{generic-a} is indeed satisfied by the residual $r\upn$ of \eqref{sys2}. To start with, \eqref{generic-a} holds at $n=0$ die to \eqref{sys2d}. Suppose now that \eqref{generic-a} is true at some $n\geq0$. Using \eqref{sys2a} together with \eqref{generic-a}, we have
$$Q^*(Qu\upnp\!-w)=Q^*(Qu\upn\!-w)+Q^*Q(\alpha_n p\upn)=r\upn+\alpha_n\'Q^*Q\'p\upn=r\upnp$$
from \eqref{sys2b}. Thus, \eqref{generic-a} holds for all $n\geq0$. It remains to determine the equivalent mapping $J_n$.

\begin{proposition}\label{prop:admissible-CG}
The sequence $(u\upn)\n0$ recursively defined by the system \eqref{sys2} for some initial iterate $u\up0$ is equivalently generated by the recurrence \eqref{generic}
where $\scS_n:=\big(\ran(Q^*)\backslash\{0\}\big)^2$ for all $n\geq1$, and for all $r\in\ran(Q^*)\backslash\{0\}$ and $s=(s_\r,s_\p)\in\scS_n$,
\begin{subequations}\label{CG-int}
\begin{eqnarray}
&J_n(r,s):=\big(D(r,s),S(r,s)\big),\quad n\geq1,\label{CG-int-Jn}\\
&\hspace{-1cm}D(r,s):=\big(\|r\|^2\big/\|Q\'P(r,s)\|^2\big)P(r,s),\quad
S(r,s):=\big(R(s),P(r,s)\big),\label{CG-int-DS}\\
&\hspace{-1cm}R(s):=s_\r+(\|s_\r\|^2/\|Qs_\p\|^2)\'Q^*Qs_\p,\quad
P(r,s):=-r+(\|r\|^2/\|s_\r\|^2)\'s_\p,\label{CG-int-RP}\\
&J_0(r,0):=\big({-}(\|r\|^2/\|Qr\|^2)\'r,(r,-r)\big).\label{CG-int-J0}
\end{eqnarray}
\end{subequations}
\end{proposition}
\ppnoi
The proof is given in Appendix \ref{app:admissible-CG}.
Although the full definition of the mapping $J_n$ is somewhat involved, the important point is simply that such a mapping exists.

\subsection{Consequences of residual extinction}\label{subsec:res-extinct}

At the level of generality level of admissible iterations, requiring residual extinction is not sufficient to guarantee Property \ref{obj}. The following proposition shows which part of this property is guaranteed.

\begin{proposition}\label{prop:basic-conv}
For an admissible iteration and a given pair $(w,u\up0)\in\hat U{\times}U$, assume that $r\upn=Q^*(Qu\upn\!-w)$ converges to zero.
\begin{enumerate}[label=(\roman*)]
\setlength{\itemsep}{0.2em}
\item If $(u\upn)\n0$ converges, then $\lim\limits_{n\to\infty}u\upn=P_{M_{Q,w}}u\up0$.
\item If $M_{Q,w}=\emptyset$, then $\lim\limits_{n\to\infty}\|u\upn\|=\infty$.
\end{enumerate}
\end{proposition}

\begin{proof}
\iti{i} If  $(u\upn)\n0$ converges with limit $u\up\infty$, we obtain $Q^*(Q u\up\infty-w)$ since $r\upn$ converges to zero. Thus, $u\up\infty\in M_{Q,w}$ by Proposition \ref{prop:MQw}\iti{iii}. Now, by assumption on $J_n$, $d_n\in\ran(Q^*)$ for all $n\geq0$. This implies that $u\up\infty\!-u\up0\in\overline{\ran(Q^*)}=\nul(Q)^\perp$. Now, by Proposition \ref{prop:MQw}\iti{iv}, $M_{Q,w}=\nul(Q)+u\up\infty$. Thus, $u\up\infty=P_{M_{Q,w}}u\up0$.

\iti{ii} Assume that $(u\upn)\n0$ is bounded. Then, there exists a subsequence $(u\up{n_k})_{k\geq0}$ that converges weakly to some $u^*\in U$. Since $Q^*Q$ is bounded,  then $(Q^*Qu\up{n_k})_{k\geq0}$ converges weakly to $Q^*Q u^*$. But as $(r\up{n_k})_{k\geq0}$ strongly converges to zero by assumption, then $(Q^*Qu\up{n_k})_{k\geq0}$ strongly converges to $Q^*w$ and hence weakly. By uniqueness of weak limits, we have $Q^*Q u^*=Q^*w$. Thus, $M_{Q,w}\neq\emptyset$ from Proposition \ref{prop:MQw}\iti{iii}.  This proves \iti{ii}.
\end{proof}
\ppnoi
What remains to be established is the mere guarantee of iteration convergence in the case $M_{Q,w}\neq\emptyset$.

\subsection{Translation equivariance}\label{subsec:trans-equiv}

As shown in the previous section, residual extinction is not sufficient to obtain the convergence of an admissible iteration in the feasible case $M_{Q,w}\neq\emptyset$. In this section, we establish sufficient conditions that guarantee this convergence as required in Property \ref{obj}\iti{i}.

\subsubsection{Core result}

\begin{theorem}\label{theo:equivar}
Let $Q\in\calB(U,\hat U)$, $U_n$ be the $n$th iterate map of an admissible iteration and $E_n$ be the associated error map. Let $(w,u\up0)$ and $(\tilde w,\tilde u\up0)$ be two pairs of $\hat U{\times}U$ such that
\begin{eqnarray}
&Q^*(Q u\up0\!- w)=Q^*(Q\tilde u\up0\!-\tilde w).\label{generic2-init}
\end{eqnarray}
Then,
\begin{equation}
U_n(\tilde w,\tilde u\up0)-U_n(w,u\up0)=\tilde u\up0\!-u\up0\quad\text{and}\quad
R_n(\tilde w,\tilde u\up0)=R_n(w,u\up0),\quad n\geq0.\label{Un-shift}
\end{equation}
Moreover, if $w$ or  $\tilde w$ is in $\dom(Q^\dagger)$, then
\begin{equation}\label{trans-imply}
E_n(w,u\up0)=E_n(\tilde w,\tilde u\up0),\qquad n\geq0.
\end{equation}
\end{theorem}

\begin{proof}
Let $(w,u\up0),(\tilde w,\tilde u\up0)\in\hat U{\times}U$ be pairs that satisfy \eqref{generic2-init}. Let $u_n:=U_n(w,u\up0)$ and  $\tilde u_n:=U_n(\tilde w,\tilde u\up0)$ for all $n\geq0$. Explicitly, let \eqref{generic} denote the recurrence satisfied by $(u\upn)\n0$ and
\begin{subequations}\label{generic2}
\begin{align}
\tilde r\upn&=Q^*(Q\tilde u\upn\!-\tilde w),\label{generic2-a}\\
(\tilde d\upn,\tilde s\upnp)&=J_n(\tilde r\upn,\tilde s\upn),\label{generic2-b}\\
\tilde u\upnp&=\tilde u\upn+\tilde d\upn,\qquad n\geq0\label{generic2-c}
\end{align}
\end{subequations}
be the recurrence satisfied by $(\tilde u\upn)\n0$.
For any $n\geq0$, we first have from \eqref{generic-a} and  \eqref{generic2-a} the implication
\begin{equation}\label{imply}
\tilde u\upn\!-u\upn=\tilde u\up0\!-u\up0\qquad\Longrightarrow\qquad \tilde r\upn=r\upn
\end{equation}
since $\tilde r\upn\!-r\upn=Q^*\big(Q(\tilde u\upn\!-u\upn)-(\tilde w-w)\big)=Q^*\big(Q(\tilde u\up0\!-u\up0)-(\tilde w-w)\big)=0$ due to \eqref{generic2-init}. The identities
\begin{equation}\label{induc}
 \tilde s\upn=s\upn \et \tilde u\upn\!-u\upn=\tilde u\up0\!-u\up0
\end{equation}
hold at $n=0$. Suppose that \eqref{induc} is true at some $n\geq0$.  Applying successively \eqref{generic2-b}, \eqref{induc}, \eqref{imply} and \eqref{generic-b}, we obtain $(\tilde d\upn,\tilde s\upnp)=J_n(\tilde r\upn,\tilde s\upn)=J_n[Q](r\upn,s\upn)=(d\upn,s\upnp)$. Hence, $\tilde s\upnp=s\upnp$ and $\tilde d\upn=d\upn$. It then follows from \eqref{generic-c} and \eqref{generic2-c} that $\tilde u\upnp\!-u\upnp=\tilde u\upn\!-u\upn=\tilde u\up0\!-u\up0$. This proves \eqref{induc} for all $n\geq0$. The second relation is precisely \eqref{Un-shift}.

Next, observe that \eqref{generic2-init} implies that $Q^*\tilde w=Q^*w-Q^*Qu^*$ with $u^*:=u\up0-\tilde u\up0$, it results from Proposition \ref{prop:MQw}\iti{iii} that the following properties are equivalent: $u\in M_{Q,\tilde w}$, $Q^*Qu=Q^*\tilde w=Q^*w-Q^*Qu^*$, $Q^*Q(u{+}u^*)=Q^*w$, $u+u^*\in M_{Q,w}$, $u\in M_{Q,w}-u^*$. Thus, $M_{Q,\tilde w}=M_{Q,w}-u^*$. Clearly, $M_{Q,w}\neq\emptyset$ $\Leftrightarrow$  $M_{Q,\tilde w}\neq\emptyset$. Assume that these equivalent properties hold. We have $P_{M_{Q,\tilde w}}\tilde u\up0=P_{M_{Q,w}-u^*}(u\up0\!-u^*)=P_{M_{Q,w}}u\up0-u^*$ by Lemma \ref{lem:Bauschke}. Then, $P_{M_{Q,\tilde w}}\tilde u\up0-P_{M_{Q,w}}u\up0=\tilde u\up0\!-u\up0=\tilde u\upn\!-u\upn$ for all $n\geq0$ from \eqref{induc}. Combining this with \eqref{imply} yields \eqref{trans-imply}.
\end{proof}

\subsubsection{Reduction to zero initial iterate}

\begin{corollary}\label{corol:trans-equiv1}
Assume the conditions of Theorem \ref{theo:equivar}. For any $(w,u\up0)\in\dom(Q^\dagger){\times}U$,
\begin{equation}\label{eq:trans-equiv1}
E_n(w,u\up0)=E_n(\tilde w,0),\quad n\geq0\where\tilde w:=w-Qu\up0\in\dom(Q^\dagger).
\end{equation}
\end{corollary}

\begin{proof}
Let $(w,u\up0)$ be any given pair in $\dom(Q^\dagger){\times}U$.
Condition \eqref{generic2-init} is clearly satisfied with $(\tilde w,\tilde u\up0):=(w{-}Qu\up0,0)$. Theorem \ref{theo:equivar} then yields $E_n(w,u\up0)=E_n(\tilde w,0)$ for all $n\geq0$. Furthermore, $\tilde w\in\dom(Q^\dagger)+\ran(Q)=\ran(Q)\oplus\ran(Q)^\perp=\dom(Q^\dagger)$ from \eqref{domQ+}.
\end{proof}
\ppnoi
As an immediate application,  Theorem \ref{theo0}\iti{i} which is limited to $u\up0=0$ can be generalized to any $u\up0\in U$.

\begin{proposition}
Let $Q$ be a nonexpansive linear operator and $(u\upn)_{n\geq0}$ be obtained by the Landweber iteration in \eqref{Landweber} with $\alpha=1$ starting from any $u\up0\in U$.
If $M_{Q,w}\neq\emptyset$, $u\upn$ converges to $P_{M_{Q,w}}u\up0$.
\end{proposition}

\begin{proof}
The Landweber iteration in \eqref{Landweber} with $\alpha=1$ is an admissible iteration. Let $U_n$ be the $n$th iterate map of this iteration and $E_n$ be the associated error map. Theorem \ref{theo0}\iti{i} implies that, for any $\tilde w\in\dom(Q^\dagger)$, $\lim_{n\to\infty}U_n(\tilde w,0)=P_{M_{Q,\tilde w}}0$, i.e., $\lim_{n\to\infty}E_n(\tilde w,0)=0$. Let $(w,u\up0)$ be any pair in $\dom(Q^\dagger){\times}U$. We conclude by Corollary \ref{corol:trans-equiv1} that $\lim_{n\to\infty}E_n(w,u\up0)=0$. Thus, the iterate $u\upn:=U_n(w,u\up0)$ converges to $P_{M_{Q,w}}u\up0$.
\end{proof}

\subsubsection{Reduction to homogeneous problem in $\nul(Q)^\perp$}

\begin{corollary}\label{corol:trans-equiv2}
Assume the conditions of Theorem \ref{theo:equivar}.
For any $(w,u\up0)\in\dom(Q^\dagger){\times}U$, there exists $\tilde u\up0\in\nul(Q)^\perp$ such that
\begin{subequations}\label{eq:trans-equiv2}
\begin{align}
R_n(w,u\up0)&=R_n(0,\tilde u\up0),\label{eq:trans-equiv2a}\\
E_n(w,u\up0)&=U_n(0,\tilde u\up0)\in\nul(Q)^\perp,\qquad n\geq0.\label{eq:trans-equiv2b}
\end{align}
\end{subequations}
\end{corollary}

\begin{proof}
Let $(w,u\up0)\in\dom(Q^\dagger){\times}U$. Let $(\tilde w,\tilde u\up0):=(0,u\up0{-}u^*)$ where $u^*:=P_{M_{Q,w}}u\up0\in M_{Q,w}$. Since $Q^*(Q\tilde u\up0-\tilde w)=Q^*Qu\up0\!-Q^*Qu^*=Q^*(Qu\up0\!-w)$ since $Q^*Qu^*=Q^*w$ from Proposition \ref{prop:MQw}\iti{iii}, \eqref{generic2-init} is satisfied. Theorem \ref{theo:equivar} then gives $R_n(w,u\up0)=R_n(0,\tilde u\up0)$, which proves \eqref{eq:trans-equiv2a}, and
\begin{equation}\label{EEU}
E_n(w,u\up0)=E_n(0,\tilde u\up0)=U_n(0,\tilde u\up0)-P_{M_{Q,0}}\tilde u\up0
\end{equation}
for all $n\geq0$. We next show that $\tilde u\up0\in\nul(Q)^\perp$. We have $\tilde u\up0=u\up0{-}P_{M_{Q,w}}u\up0$. By \eqref{PM-Qdw} that $\tilde u\up0=u\up0-Q^\dagger w-P_{\nul(Q)}u\up0=P_{\nul(Q)^\perp}u\up0-Q^\dagger w\in\nul(Q)^\perp$.
Since $0\in M_{Q,0}$, it follows from Proposition \ref{prop:MQw}\iti{iv} that $M_{Q,0}=\nul(Q)$. Hence, $P_{M_{Q,0}}\tilde u\up0=0$. This establishes the equality of \eqref{eq:trans-equiv2b}.

Finally, we show that $\tilde u\upn:=U_n(0,\tilde u\up0)\in\nul(Q)^\perp$. Equation \eqref{generic-c} implies that $\tilde u\upnp=\tilde u\upn\!+d\upn$ where $d\upn\in\ran(Q^*)$ by the definition of $J_n$. By induction, $\tilde u\upn\!-\tilde u\up0\in\ran(Q^*)\subset\nul(Q)^\perp$ while $\tilde u\up0\in\nul(Q)^\perp$. Therefore, $\tilde u\upn\in\nul(Q)^\perp$ for all $n\geq0$.
\end{proof}

\subsection{Closed range and linear convergence}\label{subsec:closed-lin}

As shown in Section \ref{subsec:closed-range}, the closedness of $\ran(Q)$ ensures that
$\Phi_{Q,w}$ admits minimizers for every $w\in\hat U$. The same
assumption naturally arises when seeking linear convergence. Indeed, in the
standard algorithms considered in this paper, closed range is precisely the
stability condition that allows residual estimates to be converted into
error estimates. We now explain this point in the framework of admissible
residual-state iterations.

Fix an admissible iteration, and let $(w,u_0)\in\dom(Q^\dagger){\times}U$ and set
$u\upn:=U_n(w,u\up0)$, $r\upn:=R_n(w,u\up0)$ and $e\upn:=E_n(w,u\up0)$. By Corollary \ref{corol:trans-equiv2}, there exists $\tilde u\up0\in\nul(Q)^\perp$ such that $r\upn=R_n(0,\tilde u\up0)$ and $e\upn=U_n(0,\tilde u\up0)\in\nul(Q)^\perp$ for all $n\geq0$. Applying \eqref{URn-rel} with 0 in place of $w$, we obtain $r_n=Q^*Q e\upn$ for all $n\geq0$. As $\gamma(Q^*Q)=\gamma(Q)^2$ \cite{Kulkarni20} and $e\upn\in\nul(Q)^\perp=\nul(Q^*Q)^\perp$, then
$$\|r\upn\|\geq\gamma(Q)^2\'\|e\upn\|,\qquad n\geq0.$$
Recall from Theorem \ref{theo:Kato} that $\gamma(Q)>0$ if and only if $\ran(Q)$ is closed.
When $\gamma(Q)>0$, any geometric decay estimate for the residuals immediately yields a
geometric decay bound for the effective errors. Otherwise, residual
extinction alone no longer provides uniform geometric control of
the iterates. This does not prove that closed range is necessary for every
possible algorithm, but it explains qualitatively why it is the natural structural
assumption under which standard residual-state iterations can reasonably be expected
to admit linear convergence estimates.

\section{Landweber iteration with variable step sizes}\label{sec:Land-var}

For any $Q\in\calB(U,\hat U)$, we investigate the Landweber iteration with varying step size presented in \eqref{Landweber-var}, which we recall here for convenience:
\begin{equation}\label{Landweber-var2}
u\upnp=u\upn\!-\alpha_n\'Q^*(Qu\upn\!-w),\qquad n\geq0.
\end{equation}

\subsection{Strong convergence}\label{subsec:LW-strong}

The goal of this section is to prove the following result.

\begin{theorem}\label{theo:Land-var}
Let $Q\in\calB(U,\hat U)$, $w\in\hat U$, $u\up0\in U$ and $(u\upn)_{n\geq0}$ satisfy the Landweber iteration \eqref{Landweber-var2}. Then, Property \ref{obj} is achieved with any sequence of relaxation coefficients
\begin{eqnarray}
&(\alpha_n)\n0\in\middfrac{2}{\|Q\|^2}\'\scR\label{Land-rel-cond}\\
\text{where}\qquad\qquad~&&\qquad\qquad\qquad\qquad\nonumber\\
&\scR:=\Big\{(\lambda_n)_{n\geq0}\in[0,1]^\NN:\smallsum{n\geq0}\lambda_n(1\!-\!\lambda_n)=\infty\Big\}.
\label{R}
\end{eqnarray}
\end{theorem}
\ppnoi
Note that \eqref{Land-rel-cond} simply means that $(\alpha_n\|Q\|^2/2)\n0\in\scR$.
We saw in Section \ref{subsub:ex-Land-var} that \eqref{Landweber-var2} is an admissible iteration. We will therefore use the results of Sections \ref{subsec:res-extinct} and \ref{subsec:trans-equiv} for the proof of Theorem \ref{theo:Land-var}.

\subsubsection{Recurrence of residual and iterate error}\label{subsub:res-err}

The key to proving Property \ref{obj} will be the observation that
\begin{equation}\label{res-err}
r\upn=Q^*(Qu\upn\!-w) \et e\upn:=u\upn\!-P_{M_{Q,w}}u\up0,\qquad n\geq0
\end{equation}
satisfy the same recurrence. While $r\upn$ is always defined, note that $e\upn$ exists only when $M_{Q,w}$ is nonempty, or equivalently when $w\in\dom(Q^\dagger)$.

\begin{proposition}
For all $n\geq0$,
\begin{subequations}\label{re-rec}
\begin{align}
r\upnp&=r\upn-\alpha_n\'Q^*Q\'r\upn,\qquad\text{with } r\up0\in\nul(Q)^\perp,\label{r-rec}\\
e\upnp&=e\upn-\alpha_n\'Q^*Q\'e\upn,\qquad\text{with } e\up0\in\nul(Q)^\perp.\label{e-rec}
\end{align}
\end{subequations}
\end{proposition}

\begin{proof}
By definition of $r\upn$, \eqref{Landweber-var2} implies that $u\upnp=u\upn-\alpha_n\'r\upn$. Then, $r\upnp=Q^*\big(Q(u\upn{-}\alpha_n r\upn)-w\big)=r\upn-\alpha_n\'Q^*Q r\upn$. Meanwhile, $r\up0\in\ran(Q^*)\subset\nul(Q)^\perp$. This proves \eqref{r-rec}.

Assuming that  $w\in\dom(Q^\dagger)$, \eqref{e-rec} can be proved by showing that every element of $M_{Q,w}$ is a fixed point of the recurrence \eqref{Landweber-var2}, while $u\upn$ remains in $u\up0\!+\ran(Q^*)\subset u\up0\!+\nul(Q)^\perp$. Instead, we use the more general results of Section \ref{subsec:trans-equiv}. It follows from \eqref{eq:trans-equiv2b} that $e\upn=E_n(w,u\up0)=U_n(0,\tilde u\up0)$ for some $\tilde u\up0\in\nul(Q)^\perp$, where $U_n$ is the $n$th iterate map of \eqref{Landweber-var2}. By definition of $U_n(\cdot,\cdot)$, $e\upn=U_n(0,\tilde u\up0)$ satisfies the recurrence \eqref{Landweber-var2} with $w=0$ starting from $e\up0=\tilde u\up0$. This proves \eqref{e-rec}.
\end{proof}
\ppnoi
In the next section, we present a result on the Krasnoselskii–Mann iteration \cite{Bartz26} that will be crucial to the convergence analysis of $r\upn$ and $e\upn$.

\subsubsection{Krasnoselskii–Mann iteration}

The following result is adapted from Theorem 2.3 of \cite{Bartz26}.

\begin{theorem}\cite{Bartz26}\label{theo:KM}
Let $T$ be a nonexpansive linear operator on $U$ and $(v\upn)\n0$ satisfy the recurrence
 \begin{equation}\label{KM-iter}
u\upnp=u\upn+\lambda_n\'(Tu\upn\!-u\upn),\qquad n\geq0.
\end{equation}
If $(\lambda_n)\n0\in\scR$, then $\lim\limits_{n\to\infty}v\upn=P_F\'v\up0$ where $F$ is the space of fixed points of $T$.
\end{theorem}

\ppnoi
The following corollary is a consequence.

\begin{corollary}\label{corol:KM}
Let $Q\in\calB(U,\hat U)$, $(\alpha_n)\n0\in\frac{2}{\|Q\|^2}\'\scR$, $u\up0\in\nul(Q)^\perp$ and $(u\upn)_{n\geq0}$ satisfy the recurrence
\begin{equation}\label{v-rec}
v\upnp=v\upn-\alpha_n\'Q^*Qv\upn,\qquad n\geq0.
\end{equation}
Then, $\lim\limits_{n\to\infty}u\upn=0$.
\end{corollary}

\begin{proof}
Iteration \eqref{v-rec} can be presented in the form \eqref{KM-iter} with
\begin{equation}\label{T-alpha}
T:=I-\middfrac{2}{\|Q\|^2}Q^*Q \et \lambda_n:=\middfrac{\|Q\|^2}{2}\'\alpha_n,\quad n\geq0.\end{equation}
Since $(\alpha_n)\n0\in\frac{2}{\|Q\|^2}\'\scR$ by assumption, $(\lambda_n)\n0\in\scR$.
Let us show that $T$ is nonexpansive. Since $T$ is self-adjoint, it is known \cite[\S2.13]{Conway90} that
\begin{equation}\label{T- norm}
\|T\|=\sup_{u\in U,\|u\|=1}|\langle u,Tu\rangle|.
\end{equation}
We have $\langle u,Tu\rangle=\|u\|^2-(2/\|Q\|^2)\|Qu\|^2$ where
$0\leq(2/\|Q\|^2)\|Qu\|^2\leq 2\|u\|^2$.
This implies that $-\|u\|^2\leq\langle u,Tu\rangle\leq\|u\|^2$. It then follows from \eqref{T- norm} that $\|T\|\leq 1$. Theorem \ref{theo:KM} then implies that $u\upn$ converges to $P_F\'u\up0$ where $F=\nul(Q^*Q)=\nul(Q)$. As $u\up0\in\nul(Q)^\perp$, $u\upn$ then converges to $P_F\'u\up0=0$.
\end{proof}

\subsubsection{Conclusion}

Combining \eqref{re-rec} with Corollary \ref{corol:KM} yields
\begin{enumerate}[label=(\alph*)]
\setlength{\itemsep}{0.2em}
\item $\lim\limits_{n\to\infty}r\upn=0$,
\item $\lim\limits_{n\to\infty}e\upn=0$ when $M_{Q,w}\neq\emptyset$.
\end{enumerate}
Property \ref{obj}\iti{ii} now follows from (a) due to Proposition \ref{prop:basic-conv}\iti{ii}. Statement (b) is precisely Property \ref{obj}\iti{i}. This proves Theorem \ref{theo:Land-var}.

\subsubsection{Prior work}

Under the conditions of Theorem \ref{theo:Land-var}, Property \ref{obj}\iti{i} was first established in \cite{qu2009landweber} in the special case where $Q$ is compact. In that work, the limit of $u\upn$ is presented in the equivalent form provided by \eqref{PM-Qdw} and the set $\scR$ is described as
\begin{equation}\label{R'}
\scR=\Big\{(\lambda_n)_{n\geq0}\in[0,1]^\NN:\smallsum{n\geq0}\min(\lambda_n,1\!-\!\lambda_n)=\infty\Big\}.
\end{equation}
This coincides with \eqref{R} as one can verify that $\mu_n/2\leq \lambda_n(1{-}\lambda_n)\leq\mu_n$ with $\mu_n:=\min(\lambda_n,1{-}\lambda_n)$ (note that $\lambda_n(1{-}\lambda_n)=\mu_n(1{-}\mu_n)$ since $\mu_n$ is either $\lambda_n$ or $1{-}\lambda_n$ and use the fact that $1{-}\mu_n\in[1/2,1]$). The compactness of $Q$ allows the authors of \cite{qu2009landweber} to prove Property \ref{obj}\iti{i} by spectral decomposition of a self-adjoint operator. This approach could actually be generalized to the non-compact case by means of spectral integration using projection-valued measures. However, by using the result of Theorem \ref{theo:KM}, the present proof of Theorem \ref{theo:Land-var} only involves elementary norm manipulations.

\subsection{Linear convergence}\label{subsec:Land-lin}

In this section, we look for stronger conditions than those of Theorem \ref{theo:Land-var} to obtain the linear convergence of $u\upn$. As motivated in Section \ref{subsec:closed-lin}, a first structural requirement will be to ensure that $\gamma(Q)>0$. Under this condition, we recall from Section \ref{subsub:closed-range} that $M_{Q,w}$ is nonempty for all $w\in\hat U$.
There will also be stricter requirements on the coefficients $(\alpha_n)\n0$, which involve the sets
\begin{align}
\scR_\eps&:=\Big[\midfrac{\eps}{2},1\!-\!\midfrac{\eps}{2}\Big]^\NN,\qquad\eps\in(0,1]\label{R-eps}\\
\scR_\lin&:=
\Big\{(\lambda_n)_{n\geq0}\in[0,1]^\NN:\liminf_{n\to\infty}\midfrac{1}{n}\'
\textstyle\sum\limits_{j\geq0}^{n-1}\min(\lambda_j,1\!-\!\lambda_j)>0\Big\}.\label{Rlin}
\end{align}
With \eqref{R'}, note that
$$\scR_\eps~\varsubsetneq~\scR_\lin~\varsubsetneq~\scR,\qquad \eps\in(0,1].$$
The present section is devoted to the proof of the following result.

\begin{theorem}\label{theo:Land-lin}
Let $Q\in\calB(U,\hat U)$ such that $\gamma(Q)>0$, $w\in\hat U$, $u\up0\in U$, $(u\upn)_{n\geq0}$ satisfy the recurrence \eqref{Landweber-var2} and
$e\upn:=u\upn-P_{M_{Q,w}}u\up0$ for all $n\geq0$.
\begin{enumerate}[label=(\roman*)]
\setlength{\itemsep}{0.4em}
\item If $(\alpha_n)\n0\in\middfrac{2}{\|Q\|^2}\'\scR_\lin$, $e\upn$ linearly converges to zero.
\item If moreover $(\alpha_n)\n0\in\middfrac{2}{\|Q\|^2}\'\scR_\eps$ for some $\eps\in(0,1]$, then
\begin{equation}\label{linrate}
\|e\upn\|\leq{(1\!-\!\eps/\kappa^2)}^n\'\|e\up0\|,\qquad n\geq0
\end{equation}
where $\kappa:=\|Q\|/\gamma(Q)$.
\end{enumerate}
\end{theorem}

\subsubsection{Contraction analysis}

Since we are in the setting where $\dom(Q^\dagger)=\hat U$ by Theorem \ref{theo:Kato}, the error $e\upn$ is always defined. From \eqref{e-rec}, an induction argument shows that $e\upn\in\nul(Q)^\perp$ for all $n\geq0$ since $\ran(Q^*)\subset\nul(Q)^\perp$. Thus, \eqref{e-rec} can be written as
\begin{eqnarray}
&e\upnp= L_{\alpha_n}\'e\upn,\qquad n\geq0\label{en-rec0}\\
\text{where}\qquad\qquad\qquad\quad~&&\qquad\qquad\qquad\qquad\quad\nonumber\\
& L_\alpha:=\big(I-\alpha\'Q^*Q\big)\big|_{\nul(Q)^\perp},\qquad \alpha\in\RR.\label{Lalpha}
\end{eqnarray}
Therefore,
\begin{eqnarray}\label{en-norm}
\|e\upnp\|\leq\| L_{\alpha_n}\|\,\|e\upn\|,\qquad n\geq0.
\end{eqnarray}
A sufficient condition for the linear convergence of $e\upn$ to zero is to have $\sup\n0\|L_{\alpha_n}\|<1$. However, a weaker condition is to require only the linear convergence of $\prod_{j=0}^n\|L_{\alpha_j}\|$ to zero. In the next three subsections, we derive the conditions on $(\alpha_n)\n0$ to achieve these conditions.

\subsubsection{A spectral norm formula}

The evaluation of $\| L_\alpha\|$ is closely related to the {\em self-adjoint} property of $ L_\alpha$. Indeed, $I-\alpha\'Q^*Q$ is self-adjoint on $U$ and leaves $\nul(Q)^\perp$ invariant since $\ran(Q^*)\subset\nul(Q)^\perp=\nul(Q)^\perp$.
The expression for  $\| L_\alpha\|$ is obtained using arguments similar to those used in the analysis of the frame algorithm \cite{Duffin52,Grochenig93}.

\begin{proposition}\label{prop:rho}
For any $\alpha\geq0$,
\begin{equation}\label{rho-bound}
\| L_\alpha\|=\max\big(1-\alpha\gamma(Q)^2,\alpha\|Q\|^2-1\big).
\end{equation}
\end{proposition}

\begin{proof}
As $L_\alpha$ is a self-adjoint operator on $\nul(Q)^\perp$, we recall from \eqref{T- norm} that
\begin{equation}\label{Lalpha-norm}
\| L_\alpha\|=\sup_{e\in\nul(Q)^\perp,\|e\|=1}\big|\langle e, L_\alpha e\rangle\big|.
\end{equation}
Let $e$ be any element of $\nul(Q)^\perp$ of norm 1. It follows from \eqref{gamma0} that
$$\gamma(Q)^2\,\leq\,\|Qe\|^2\,\leq\,\|Q\|^2$$
where the bounds are tight. Since $\langle e, L_\alpha e\rangle=\langle e,e\rangle-\alpha\langle e,Q^*Qe\rangle=1-\alpha\|Qe\|^2$, then
\begin{equation}\label{L-ineq}
1-\alpha\|Q\|^2~\leq~\langle e, L_\alpha e\rangle~\leq~1-\alpha\gamma(Q)^2
\end{equation}
where the bounds are again tight. After verifying that $a\leq b\leq c$ implies that  $|b|\leq\max(c,-a)$, we obtain \eqref{rho-bound} from \eqref{Lalpha-norm} and \eqref{L-ineq}.
\end{proof}

\subsubsection{Uniform bounding of $\|L_{\alpha_n}\|$}

Our goal here is to find a condition on $(\alpha_n)\n0$ so that $\|L_{\alpha_n}\|$ has a uniform upper bound that is less than 1. For that purpose, we will use the following lemma.

\begin{lemma}\label{lem:basic-cond}
Assume that $0\leq a \leq b\leq2$ with $b\neq0$. Then,
\begin{equation}\label{max-ab}
\max(1\!-\!a,b\!-\!1)\leq 1-\frac{a}{b}\'\min(b,2\!-\!b).
\end{equation}
\end{lemma}

\begin{proof}
Let $\mu:=\min(b,2{-}b)$. On the one hand,
$b\geq\mu$ implies that $1{-}a\leq1{-}\frac{\mu}{b}\'a=1{-}\frac{a}{b}\'\mu$. On the other hand, $b\leq\max(b,2{{-}}b)= 2{-}\mu$ implies that $ b{-}1\leq1{-}\mu\leq1{-}\frac{a}{b}\'\mu$ since $\frac{a}{b}\leq1$.
\end{proof}
\ppnoi
We can then apply the inequality \eqref{max-ab} to bound $\| L_{\alpha_n}\|$ based on in \eqref{rho-bound}. We obtain the following result.

\begin{proposition}\label{prop:Lalpha-bound}
Assume that $\gamma(Q)>0$ and $(\alpha_n)\n0\in\middfrac{2}{\|Q\|^2}\'\scR_\eps$ for some
$\eps\in(0,1]$, where $\scR_\eps$ is defined in \eqref{R-eps}. Then,
\begin{equation}\label{lin-Land-decay}
\| L_{\alpha_n}\|\leq 1-\eps/\kappa^2,\quad n\geq0
\where \kappa:=\|Q\|/\gamma(Q).
\end{equation}
\end{proposition}

\begin{proof}
By assumption on $(\alpha_n)\n0$, $\alpha_n\|Q\|^2/2\in [\frac{\eps}{2},1{-}\frac{\eps}{2}]$ for every $n\geq0$, which implies that $\alpha_n\|Q\|^2\in[\eps,2{-}\eps]$. Let $n\geq0$.
From \eqref{rho-bound}, we can write that
$\| L_{\alpha_n}\|=\max(1{-}a,b{-}1)$ with $a=\alpha_n\'\gamma(Q)^2$ and $b=\alpha_n\|Q\|^2$. As the conditions of Lemma \ref{lem:basic-cond} are satisfied, \eqref{max-ab} implies that $\| L_{\alpha_n}\|\leq 1-\frac{a}{b}\'\min(b,2{-}b)$. Since $b=\alpha_n\|Q\|^2\in[\eps,2{-}\eps]$, then $\min(b,2{-}b)\geq\eps$ and hence $\| L_{\alpha_n}\|\leq 1-\frac{a}{b}\'\eps= 1-\frac{\gamma(Q)^2}{\|Q\|^2}\'\eps=1-\eps/\kappa^2$.
\end{proof}
\ppnoi
This proves Theorem \ref{theo:Land-lin}\iti{ii} in view of \eqref{en-norm}.

\subsubsection{A weaker condition of linear convergence}

As opposed to $\scR$, the elements $(\lambda_n)\n0$ of $\scR_\eps$ are such that $\lambda_n$ cannot be arbitrarily close to 0 or 1. This constraint is relaxed in the following proposition.

\begin{proposition}\label{prop:Land-lin2}
If $\gamma(Q)>0$ and  $(\alpha_n)\n0\in\frac{2}{\|Q\|^2}\'\scR_\lin$, then $\prod_{j=0}^{n-1}\| L_{\alpha_j}\|$ linearly converges to zero.
\end{proposition}

\begin{proof}
We begin as in the proof of Proposition \ref{prop:Lalpha-bound}.
From \eqref{rho-bound},
$\| L_{\alpha_n}\|=\max(1{-}a_n,b_n{-}1)$ with $a_n=\alpha_n\'\gamma(Q)^2$ and $b_n=\alpha_n\|Q\|^2$. By applying \eqref{max-ab}, $\| L_{\alpha_n}\|\leq 1-\frac{a_n}{b_n}\'\min(b_n,2{-}b_n)=1-\frac{2}{\kappa^2}\'\min(c_n,1{-}c_n)$ where $\kappa:=\|Q\|/\gamma(Q)$ and $c_n:=b_n/2=\alpha_n\|Q\|^2/2$ for all $n\geq0$.
Using the elementary inequality $1-x\leq e^{-x}$, we then have
$$\textstyle\prod\limits_{j=0}^{n-1}\| L_{\alpha_n}\|\leq
\exp\Big(-\midfrac{2}{\kappa^2}\'\sum\limits_{j=0}^{n-1}\min(c_n,1{-}c_n)\Big).$$
Since $(\alpha_n)\n0\in\frac{2}{\|Q\|^2}\'\scR_\lin$, then $(c_n)\n0\in\scR_\lin$. Therefore, there exist $N\geq1$ and $c>0$ such that
$\middfrac{1}{n}\'\textstyle\sum_{j\geq0}^{n-1}\min(c_j,1\!-\!c_j)\geq c$ for all $n\geq N$. Thus,  $\prod_{j=0}^{n-1}
\| L_{\alpha_j}\|\leq e^{-\xi n}$ for all $n\geq N$ with $\xi=2c/\kappa^2$. Thus, $\prod_{j=0}^{n-1}\| L_{\alpha_j}\|$ linearly converges to zero.
\end{proof}
\ppnoi
This justifies Theorem \ref{theo:Land-lin}\iti{i} via \eqref{en-norm}.

\section{Alternating projections with one-sided varying relaxation}\label{sec:alt-var}

In this section, we return to the original problem of alternating projections between $U$ and $W$ and study the relaxed iteration on $U$
\begin{eqnarray}
&u\upnp=P_U P_W^{\alpha_n} u\upn,\qquad n\geq0\label{POCS-rel}\\
\text{where}\qquad\qquad\quad&&\qquad\qquad\qquad\qquad~\nonumber\\
&P_W^\alpha u:=u+\alpha(P_W u-u),\qquad u\in H,~\alpha\in\RR.\label{P-rel}
\end{eqnarray}
With $(Q,w)$ defined by \eqref{Q}, we obtain
\begin{align}
P_U P_W^\alpha u&=u+\alpha\'P_U(P_W u-u),\qquad u\in U\nonumber\\
&=u+\alpha\'Q^*(Qu-w)
\end{align}
by \eqref{PW-gen} and \eqref{Q*}.
Thus, \eqref{POCS-rel} takes the form \eqref{Landweber-var2}
which is the Landweber iteration studied in Section \ref{sec:Land-var}. Consequently, by choosing $Q$ and $w$ as in \eqref{Q},  all convergence properties established in Section \ref{sec:Land-var} immediately apply to iteration \eqref{POCS-rel}. The main task is to express $\|Q\|$ and $\gamma(Q)$ involved in Theorems \ref{theo:Land-var} and  \ref{theo:Land-lin} in terms of the angles between $U$ and $V$.

\subsection{Strong convergence}

Theorem \ref{theo:Land-var} on strong convergence involves $\|Q\|$. When $Q$ is defined by \eqref{Q}, we have
\begin{equation}\label{Q-norm}
\|Q\|=\|P_{V^\perp}P_U\|=\|P_U P_{V^\perp}\|,
\end{equation}
since $\|P_{V^\perp}P_U\|=\|(P_{V^\perp}P_U)^*\|
=\|P_U^*P_{V^\perp\!}^*\|=\|P_U P_{V^\perp\!}\|$.
It is shown in \cite{Bauschke16} for finite dimensional spaces that $\arcsin\|P_U P_{V^\perp\!}\|$ is the largest principal angle between $U$ and $V$. We  extend this definition to infinite dimensional spaces by formally defining
\begin{equation}\label{theta-max}
\theta_p=\theta_p(U,V):=\arcsin\|P_U P_{V^\perp\!}\|.
\end{equation}
Since this yields the expression $\|Q\|=\sin\theta_p$, Theorem \ref{theo:Land-var} therefore yields the following result.
\line
\begin{proposition}\label{prop:AP-strong}
Let $(u\upn)\n0$ be recursively defined by the relaxed alternating projections of \eqref{POCS-rel}. Then Property \ref{prop:Bauschke} holds for any sequence of relaxation coefficients
\begin{equation}\label{rel-cond0}
(\alpha_n )_{n\geq0}\in\middfrac{2}{\sin^2\theta_p}\,\scR.
\end{equation}
\end{proposition}
\ppnoi
Since $\sin\theta_p=\sup_{u\in U\|u\|=1}\|P_{V^\perp} u\|=\sup_{u\in U\|u\|=1}\|u-P_V u\|$, $\sin\theta_p$ may be substantially smaller than 1, depending on the relative position between $U$ and $V$. In this case, $\alpha_n$ can be substantially larger than 2, which is unusual for non-adaptive alternating projections. Such re;axation parameters were initially obtained in \cite{Bauschke16} which we discuss in detail in Section \ref{subsec:lin-prior}.

\subsection{Linear convergence}\label{subsec:lin-alt}

The linear convergence of $u\upn$ in \eqref{POCS-rel} is analyzed by applying the results from Section \ref{subsec:Land-lin} to the pair $(Q,w)$ defined by \eqref{Q}. Recall that the convergence rate was obtained from the expression for $\|L_\alpha\|$ where $L_\alpha$ is defined in \eqref{Lalpha}. By \eqref{rho-bound}, $\|L_\alpha\|$ depends on $\gamma(Q)$ and $\|Q\|$. Under \eqref{Q}, we have from Section \ref{subsub:AP-angle} that $\gamma(Q)=\sin\theta_\f(U,V)$.
Since $\|Q\|=\sin\theta_p$ from \eqref{Q-norm} and \eqref{theta-max}, \eqref{rho-bound} becomes
\begin{equation}\label{rho-bound2}
\|L_\alpha\|=\max\big(1-\alpha \sin^2\theta_\f,~\alpha\sin^2\theta_p-1\big).
\end{equation}
The following proposition is a direct application of Theorem \ref{theo:Land-lin}. Under the assumption that $\theta_f(U,V)>0$, recall from Proposition \ref{prop:UW-min} that $M_{U,W}$ is always nonempty.

\begin{proposition}\label{prop:POCS-lin}
Assume that $\theta_f=\theta_f(U,V)>0$.
Let $(u\upn)\n0$ be defined recursively by the relaxed alternating projections of \eqref{POCS-rel} and $e\upn:=u\upn\!-P_{M_{U,W}}u\up0$ for all $n\geq0$.
\begin{enumerate}[label=(\roman*)]
\setlength{\itemsep}{0.4em}
\item If $(\alpha_n)\n0\in\middfrac{2}{\sin^2\theta_p}\'\scR_\lin$, $e\upn$ linearly converges to zero,
\item If moreover $(\alpha_n)\n0\in\middfrac{2}{\sin^2\theta_p}\'\scR_\eps$ for some $\eps\in(0,1]$, then
\begin{equation}\label{POCS-linrate}
\|e\upn\|\leq{(1\!-\!\eps/\kappa^2)}^n\'\|e\up0\|,\qquad n\geq0
\where \kappa:=\sin\theta_p/\sin\theta_f.
\end{equation}
\end{enumerate}
\end{proposition}

\subsection{Prior knowledge on linear convergence}\label{subsec:lin-prior}

Previous work on the linear convergence of iteration \eqref{POCS-rel} has focused on linear subspaces \cite{Deutsch1995,Bauschke16} with a constant relaxation coefficient. This section reviews these results. We then explain how the results of Section \ref{subsec:lin-alt} improve and extend the prior work.

\subsubsection{Setting of prior work}

As mentioned above, previous work focused on alternating projections on linear subspaces and constant relaxation. Using our notation, this amounts to  the iteration \begin{equation}\label{POCS-lin-const}
u\upnp=P_U P_V^\alpha u\upn,\qquad n\geq0.
\end{equation}
This corresponds to the special case of \eqref{POCS-rel} where $W=V$ and $\alpha_n=\alpha$ for all $n\geq0$. Then, $M_{U,W}=U\cap V$ and the error $e\upn:=u\upn\!-P_{M_{Q,w}}u\up0$ takes the form
\begin{equation}\label{en-lin}
e\upn=u\upn-P_{U\cap V}\'u\up0=(P_U P_V^\alpha)^n\'u\up0-P_{U\cap V}u\up0,\qquad n\geq0.
\end{equation}
Let
\begin{equation}\label{Falpha}
F_\alpha:=P_U P_V^\alpha-P_{U\cap V}.
\end{equation}
Using the identities
\begin{equation}\label{prop-Falpha}
P_{U\cap V}\'(P_U P_V^\alpha)^n=(P_U P_V^\alpha)^n P_{U\cap V}=P_{U\cap V},\qquad n\geq0,~\alpha\in\RR,
\end{equation}
which follow directly from the equalities $P_{U\cap V}\'Q=Q P_{U\cap V}=P_{U\cap V}$
with $Q=P_U$,  $Q=P_V$ and $Q=P_V^\alpha$,
\cite{Bauschke16} shows that
\begin{equation}\label{Falphan}
F_\alpha^n=(P_U P_V^\alpha)^n-P_{U\cap V},\qquad n\geq1.
\end{equation}
Thus, \eqref{en-lin} becomes
\begin{equation}\label{POCS-err-const}
e\upn=\big((P_U P_V^\alpha)^n\'u\up0-P_{U\cap V}\big)u\up0=F_\alpha^n\'u\up0,\qquad n\geq1.
\end{equation}
The linear convergence of $e\upn$ to zero can then be studied using the inequality
\begin{equation}\label{POCS-linconv}
\|e\upn\|\leq\|F_\alpha^n\|\,\|u\up0\|,\qquad n\geq1.
\end{equation}
The simple bound $\|F_\alpha^n\|\leq\|F_\alpha\|^n$ is readily available, but it is suboptimal. Previous work therefore focused on the evaluation of $\|F_\alpha^n\|$.

\subsubsection{Unrelaxed case}

The first linear convergence result for alternating projections was established for the unrelaxed case $\alpha=1$ and is as follows.
\line
\begin{lemma}\cite{Kayalar88}\label{lem:Deutsch}
\begin{equation}\label{eq:Deutsch}
\|F_1^n\|=\big\|(P_U P_V)^n\!-P_{U\cap V}\big\|=\cos^{2n-1}\theta_\f,\qquad n\geq1
\end{equation}
where $\theta_\f$ is defined in \eqref{tf}.
\end{lemma}

\subsubsection{Constant relaxation in finite dimension}\label{subsub:const-rel}

The more general case of constant relaxation was studied in \cite{Bauschke16}, but only in finite-dimensional spaces.
Whenever the spectral radius $\rho(F_\alpha)$ of $F_\alpha$ is less than 1, it is shown that
\begin{equation}\label{opt-decay}
\exists\'c>0,\;N\geq0,\qquad\|F_\alpha^n\|\leq c\,\rho(F_\alpha)^n,\qquad n\geq N.
\end{equation}
Moreover, $\rho(F_\alpha)$ is the smallest coefficient for which such a bound holds. It is then shown
for $\alpha>0$  that
\begin{equation}\label{rho}
\rho(F_\alpha)=\max\big(1-\alpha \sin^2\theta_\f,~\alpha\sin^2\theta_p-1\big)
\end{equation}
where $\theta_\f=\theta_\f(U,V)$ is defined in \eqref{tf} and $\theta_p=\theta_p(U,V)$ is the largest principal angle between $U$ and $V$.
It is noted in equ. (67) of \cite{Bauschke16} that $\sin\theta_p=\|P_U-P_U P_V\|$, which matches our definition of $\theta_p$ in \eqref{theta-max}. There is striking resemblance between \eqref{rho} and \eqref{rho-bound2}. We explain this coincidence in the next section.

\subsection{Connection of present work to prior knowledge}

\subsubsection{Difference with prior work}

The main difference between the present work and previous work is the restriction of the initial iterate $u\up0$ to $U$. This may appear to be a minor point since $u\upn\in U$ as soon as $n\geq1$ regardless of $u\up0\in H$. Indeed, this may not affect the exponent of linear convergence. However, this hides remarkable algebraic properties of the iterates $u\upn$ that already appear at $n=1$, if not at $n=0$. The algebraic difference is immediately apparent by comparing the formulas
\begin{subequations}\label{en-expl}
\begin{alignat}{3}
e\upn&=F_\alpha^n\'u\up0,\quad n\geq0,\qquad&&\text{when}\quad u\up0\in H,\label{en-expl-a}\\
e\upn&=L_\alpha^n\'e\up0,\quad n\geq0,\qquad&&\text{when}\quad u\up0\in U\label{en-expl-b}
\end{alignat}
\end{subequations}
where the first relation is from \eqref{POCS-err-const} and the second from \eqref{en-rec0}. The key difference is that $L_\alpha$ is {\em self-adjoint} as seen in \eqref{Lalpha}, while $F_\alpha$ is not. We saw how the self-adjoint property made it possible to derive an explicit expression for $\|L_\alpha\|$ in \eqref{rho-bound} and hence in \eqref{rho-bound2} in the case where $(Q,w)$ is given by \eqref{Q}. We next examine the consequences in greater detail.

\subsubsection{Comparison between $F_\alpha^n$ and $L_\alpha^n$}

To compare \eqref{en-expl-a} and \eqref{en-expl-b}, let us establish some additional properties of $F_\alpha^n$ and $L_\alpha^n$.

\begin{proposition}\label{prop:F-L}
For any $\alpha\in\RR$,
\begin{enumerate}[label=(\roman*)]
\setlength{\itemsep}{0.3em}
\item $\ran(F_\alpha)\subset U_0$,
\item $F_\alpha^n=F_\alpha^n(I-P_{U\cap V})$ for all $n\geq1$,
\item $L_\alpha=P_U P_V^\alpha|_{U_0}=F_\alpha|_{U_0}$,\footnote
{Strictly speaking, $F_\alpha|_{U_0}$ is regarded as an operator from $U_0$ into $H$, whereas $L_\alpha$ is an operator from $U_0$ into $U_0$. We identify $L_\alpha$ with its composition with the inclusion $U_0\hookrightarrow H$.}
where $U_0$ is defined in \eqref{U0V0}.
\end{enumerate}
\end{proposition}

\begin{proof}
Applying \eqref{prop-Falpha} to the expression for $F_\alpha^n$ in \eqref{Falphan}, we obtain that
\begin{equation}\nonumber
P_{U\cap V}\'F_\alpha^n=F_\alpha^n\'P_{U\cap V}=0,\qquad n\geq1.
\end{equation}
This immediately proves \iti{ii}. For \iti{i}, it already follows from \eqref{Falpha} that $\ran(F_\alpha)\subset U$. Since $P_{U\cap V}F_\alpha=0$, then $\ran(F_\alpha)\subset U\cap(U\!\cap\!V)^\perp=U_0$.

Next, $L_\alpha$ is obtained from \eqref{Lalpha} with $(Q,w)$ from \eqref{Q}. The domain of $L_\alpha$ is $\nul(Q)=U\cap V$ and hence $\nul(Q)^\perp=U\cap(U{\cap}V)^\perp=U_0$ as defined in \eqref{U0V0} (since $\nul(Q)^\perp$ designates the orthogonal complement of $\nul(Q)$ in the domain of $Q$). Thus, for all $u\in\nul(Q)^\perp=U_0$, $L_\alpha u=u -\alpha \'P_U P_{V^\perp}u =P_U(I{-}\alpha \' P_{V^\perp})u =P_U P_V^{\alpha }u $ from \eqref{P-rel} since $P_{V^\perp}=I-P_V$. This proves the first equality of \iti{iii}. Meanwhile, for any $u\in U_0$, $P_{U\cap V}\'u=0$, which implies that $F_\alpha u=P_U P_V^\alpha u$. This is the second equality of \iti{iii}.
\end{proof}
\ppnoi
An immediate consequence is the following refinement of \eqref{en-expl-a}. Since $e\up0=u\up0-P_{U\cap V}u\up0=(I-P_{U\cap V})u\up0$, it follows from \iti{ii} that $F_\alpha^n\'e\up0=F_\alpha^n\'u\up0$ for all $n\geq1$. Then, \eqref{en-expl-a} is more precisely
\begin{equation}\label{en-expl-a'}
e\upn=F_\alpha^n\'e\up0,\quad n\geq0,\qquad\text{when}\quad u\up0\in H.\tag{\ref{en-expl-a}'}
\end{equation}
The second important consequence is that
\begin{equation}\label{rho-norm}
\rho(F_\alpha)=\rho(L_\alpha)=\|L_\alpha\|,\qquad\alpha\in\RR.
\end{equation}
Indeed, while the first equality follows from Lemma \ref{lem:restr-rad} in the appendix with $R=F_\alpha$ and $S=U_0$, the second equality is precisely where the self-adjoint property of $L_\alpha$ comes into play (see Theorem VI.6 of \cite{Reed80}). This is the fundamental reason for the coincidence between \eqref{rho} and \eqref{rho-bound2}.

\subsubsection{Error decay in unrelaxed case}

Lemma \ref{lem:Deutsch} gave the exact evaluation of $\|F_\alpha^n\|$ for $n\geq1$ in the unrelaxed case $\alpha_n\equiv1$. Using \eqref{en-expl-a'}, we obtain the error bound $\|e\upn\|\leq(\cos^{2n-1}\!\theta_\f)\|e\up0\|$. To extract the convergence rate, it is convenient to present this as
\begin{equation}\label{altproj-decay}
\|e\upn\|\leq c_0\,\big(1\!-\!\sin^2\theta_\f\big)^n\|e\up0\|,\quad n\geq1
\where c_0:=\cos^{-1}\theta_\f.
\end{equation}
It is interesting to compare this with the result of Proposition \ref{prop:POCS-lin}\iti{iii}. As $\alpha_n\equiv 1$ in the present case, the condition that $(\alpha_n )_{n\geq0}\in (2\'\sin^{-2}\!\theta_p)\scR_\eps$ is equivalent to $\sin^2\!\theta_p/2\in[\eps,1{-}\eps]$. Since $\sin^2\!\theta_p/2\leq 1/2$, the largest $\eps$ that enables this condition is $\eps=\sin^2\!\theta_p/2$. Substituting this value into \eqref{POCS-linrate}, we obtain precisely
$$\|e\upn\|\leq\big(1-\sin^2\theta_\f\big)^n\,\|e\up0\|,\qquad n\geq0.$$
This matches the decay rate in \eqref{altproj-decay}, but with the smaller scaling factor of $c_0=1$ as opposed to $\cos^{-1}\theta_\f$ in \eqref{altproj-decay}. This is due to our assumption that $u\up0\in U$.

\subsubsection{Asymptotic behavior of $\|F_\alpha^n\|$}

Another important consequence of the self-adjoint property of $L_\alpha$ is that
$$\|L_\alpha^n\|=\|L_\alpha\|^n,\qquad n\geq0.$$
Proposition \ref{prop:F-L}\,\iti{i} and \iti{iii} imply that
\begin{equation}\label{FE}
F_\alpha^n=L_\alpha^{n-1}F_\alpha,\qquad n\geq1.
\end{equation}
Assuming that $L_\alpha\neq0$, this implies that
$$\|F_\alpha^n\|\leq\|L_\alpha^{n-1}\|\,\|F_\alpha\|
=c\'\|L_\alpha\|^n,\quad n\geq1\where c:=\|F_\alpha\|/\|L_\alpha\|.$$
Using the identity $\|L_\alpha\|=\rho(F_\alpha)$ from \eqref{rho-norm}, we actually have the following more precise result.

\begin{proposition}
Assume that $\rho(F_\alpha)\neq0$. Then, at the limit of $n$ going to $\infty$,
\begin{equation}\label{asym-Falpha}
\|F_\alpha^n\|\sim c_0\,\rho(F_\alpha)^n\quad\text{for some } c_0\in\big[1,\middfrac{\|F_\alpha\|}{\rho(F_\alpha)}\big].
\end{equation}
\end{proposition}

\begin{proof}
For $\alpha\in\RR$ given, we will use the notation $F=F_\alpha$ and $L=L_\alpha$ for simplicity. Assume $L\neq0$. By \eqref{FE} and Proposition \ref{prop:F-L}\iti{iii},
$F^n|_{U_0}=L^{n-1}F|_{U_0}=L^{n-1}L=L^n$ for all $n\geq1$. Hence, $\|L\|^n=\|L^n\|\leq\|F^n\|$. Thus, $\gamma_n:=\|F^n\|/\|L\|^n\geq1$ for all $n\geq1$.
It results from \eqref{FE} that $\|F^{n+1}\|=\|L^n F\|\leq\|L\|\,\|L^{n-1}F\|=\|L\|\,\|F^n\|$. Dividing both sides by $\|L^{n+1}\|$ yields $\gamma_{n+1}\leq\gamma_n$. Thus, $\gamma_n$ tends to a limit $c_0\in[1,\gamma_1]=[1,\|F\|/\|L\|]$. Using the identity $\|L\|=\rho(F)$ from \eqref{rho-norm}, we obtain \eqref{asym-Falpha}.
\end{proof}
\ppnoi
This result is more precise than \eqref{opt-decay} (in fact true even if $\rho(F_\alpha)\geq1$). However, it is even more informative to observe from \eqref{en-expl-a'} and \eqref{FE} that
\begin{equation}\label{steady}
e\upnp=L_\alpha^n\'e\up1,\quad n\geq0,\qquad\text{when}\quad u\up0\in H.
\end{equation}
Another way to derive this is to see that $e\up1=u\up1-P_{U\cap V}u\up0$ from \eqref{en-lin} is also equal to $u\up1-P_{U\cap V}u\up1$ using \eqref{prop-Falpha} with $n=1$. As $u\up1$ is in $U$, \eqref{en-expl-b} can be used with $u\up1$ as initial iterate, which leads to \eqref{steady}.
In signal processing, \eqref{steady} would correspond to the ``steady state'' of the iteration, while the initial transformation $e\up1=F_\alpha e\up0$ would be regarded as the ``transient''. A common practice is to study these two modes separately.

\subsubsection{Varying relaxation}

The self-adjointness of $L_\alpha$ also led to a practical extension of the alternating projections to variable relaxation on $P_V$,
\begin{equation}\label{POCS-lin-var}
u\upnp=P_U P_V^{\alpha_n} u\upn,\qquad n\geq0.
\end{equation}
As a generalization of \eqref{en-expl-a'} and \eqref{en-expl-b}, one verifies that
\begin{subequations}\label{en-expl-var}
\begin{alignat}{3}
e\upnp&=F_{\alpha_n}\'e\upn,\quad n\geq0,\qquad&&\text{when}\quad u\up0\in H,\label{en-expl-vara}\\
e\upnp&=L_{\alpha_n}\'e\upn,\quad n\geq0,\qquad&&\text{when}\quad u\up0\in U.\label{en-expl-varb}
\end{alignat}
\end{subequations}
For the second case, we studied the error decay in Section \ref{subsec:Land-lin} by evaluating $\prod_{j=1}^n\|L_{\alpha_n}\|$, which has the advantage of being equal to
$\big\|\textstyle\prod_{j=1}^n L_{\alpha_n}\big\|$ while reducing the analysis to the tractable derivation of $\|L_{\alpha_n}\|$. This would not be possible for $F_{\alpha_n}$ in the case of \eqref{en-expl-vara}. However, it can still be shown, as an extension of \eqref{steady}, that $e\upnp=L_{\alpha_n}\'e\upn$ is true with any $u\up0\in H$ for $n\geq1$. In other words, the case of varying relaxation remains accessible with any $u\up0\in H$ when starting the analysis from $n=1$.

\section{Steepest-descent variant of Landweber iteration}\label{sec:var}

Recall from Section \ref{subsec:standard} that the steepest-descent variant of Landweber iteration is defined by the recurrence
\begin{subequations}\label{SD-rec}
\begin{align}
r\upn&=Q^*(Qu\upn\!-w),\label{SD-rec-a}\\
u\upnp&= u\upn-\alpha_n\'r\upn,\qquad n\geq0\label{SD-rec-b}\\
\text{with}\qquad\qquad\qquad\qquad\qquad\qquad&&&\qquad\qquad\qquad\qquad\qquad\nonumber\\
\alpha_n:=\argmin_{\alpha\in\RR}&~\Phi_{Q,w}\big(u\upn\!-\alpha r\upn\big)
=\dfrac{\|r\upn\|^2}{\|Qr\upn\|^2}\label{an2}
\end{align}
\end{subequations}

\subsection{Residual extinction}\label{subsec:SD-extinct}

We show that the residual $r\upn$ necessarily converges to zero. Observe from \eqref{SD-rec-b} and \eqref{an2} that $u\upnp$ can equivalently be written as
\begin{eqnarray}
&u\upnp=\displaystyle\argmin_{u\in\calL_{Q,w}(u\upn)}\Phi_{Q,w}(u),\qquad n\geq0\label{SD-gen}\\
\text{where}\qquad\quad&&\qquad\qquad\qquad\nonumber\\
&\calL_{Q,w}(u):=u+\Span\{r\}\qquad\text{with}\qquad r:=Q^*(Qu-w).\label{LQw}
\end{eqnarray}
The set $\calL_{Q,w}(u)$ is the line through $u$ in the direction of steepest descent.

\begin{lemma}\label{lem:dPhi}
Let $Q\in\calB(U,\hat U)$ and $w\in\hat U$. For any $u\in U$,
\begin{equation}\label{dPhi}
\Phi_{Q,w}(u)-\!\min_{v\in\calL_{Q,w}(u)}\!\!\Phi_{Q,w}(v) ~\geq~\frac{\|r\|^2}{2\|Q\|^2}
\where  r:=Q^*(Qu-w).
\end{equation}
\end{lemma}

\begin{proof}
Let
$u':=\argmin_{v\in\calL_{Q,w}(u)}\!\Phi_{Q,w}(v)$, $e:=Qu-w$ and $e':=Qu'-w$
so that
$\Phi_{Q,w}(u)=\frac12\|e\|^2$ and $\Phi_{Q,w}(u')=\frac12\|e'\|^2$. Since $u'$ minimizes  $\Phi_{Q,w}$ in $u+\Span\{r\}$, then $e'$ is the minimum-norm element of
$Q\big(u{+}\Span\{r\})-w=e+\Span\{Qr\}$. Thus, $e'=P_{\Span\{Qr\}^\perp}e$. By the Pythagorean theorem, $\|e\|^2=\|e'\|^2+\|P_{\Span\{Qr\}}e\|^2$. Thus,
$$\Phi_{Q,w}(u)-\Phi_{Q,w}(u')=\midfrac12\'\big\|P_{\Span\{Qr\}}e\big\|^2
=\midfrac{|\langle e,Qr\rangle\|^2}{2\|Qr\|^2}
=\midfrac{\|r\|^4}{2\|Qr\|^2}\geq\midfrac{\|r\|^4}{2\|Q\|^2\|r\|^2}=\midfrac{\|r\|^2}{2\|Q\|^2}$$
where, in the third equality, we have used the identities $\langle e,Qr\rangle=\langle Q^*e,r\rangle=\langle r,r\rangle=\|r\|^2$.
\end{proof}

\begin{corollary}\label{SD-extinct}
Let $Q\in\calB(U,\hat U)$, $w\in\hat U$ and $(u\upn)\n0$ be the sequence in $U$ defined recursively by  \eqref{SD-rec} and \eqref{an2} . Then, $\sum\n0\|r\upn\|^2<\infty$.
\end{corollary}

\begin{proof}
It follows from \eqref{dPhi} and \eqref{SD-gen} that $\Phi_{Q,w}(u\upn)-\Phi_{Q,w}(u\upnp)\geq\|r\upn\|^2/(2\|Q\|^2)$.
Therefore,
$\sum\n0\|r\upn\|^2\leq 2\|Q\|^2\'\Phi_{Q,w}(u\up0)<\infty$.
\end{proof}
\ppnoi
Consequently, $\lim\limits_{n\to\infty}r\upn=0$.

\subsection{Strong convergence}\label{subsec:SD-conv}

The following result is addapted from Theorem 3.2 of \cite{Kammerer71}.

\begin{theorem}\cite{Kammerer71}\label{theo:Kammerer}
Let $Q\in\calB(U,\hat U)$, $w\in\hat U$, $\bar w:=P_{\,\overline{\ran(Q)}}\'w$ and $(u_n)\n0$ be recursively defined by the recurrence \eqref{SD-rec} starting from $u\up0=0$. If $\bar w\in\ran(QQ^*)$, then $u\upn$ converges to\footnote
{As in footnote \ref{foot:Q+w}, the limit of $u\upn$ is originally presented as $Q^\dagger w$.}
$P_{M_{Q,w}}0$.
\end{theorem}
\ppnoi
Note that the above condition $\bar w\in\ran(QQ^*)$ is stronger than $M_{Q,w}\neq\emptyset$ as shown in the following proposition.

\begin{proposition}\label{prop:SD-cond}
Let  $\bar w:=P_{\,\overline{\ran(Q)}}\'w$. Then,
\begin{subequations}\label{bw-cond}
\begin{eqnarray}
\bar w\in\ran(Q)&\iff& M_{Q,w}\neq\emptyset,\label{bw-conda}\\[0.5ex]
\bar w\in\ran(QQ^*)&\iff& M_{Q,w}\cap\ran(Q^*)\neq\emptyset.\label{bw-condb}
\end{eqnarray}
\end{subequations}
\end{proposition}

\begin{proof}
By \eqref{domQ+}, $M_{Q,w}\neq\emptyset$ $\Leftrightarrow$ $w\in\ran(Q)\oplus\ran(Q)^\perp$ $\Leftrightarrow$ $\bar w\in\ran(Q)$. This proves \eqref{bw-conda}. Since $\ran(QQ^*)=Q(\ran(Q^*)$, then $\bar w\in\ran(QQ^*)$ if and only if there exists $u\in\ran(Q^*)$ such that $\bar w=Qu$. This is in turn equivalent to the existence of an element $u$ in $\ran(Q^*)\cap M_{Q,w}$ by Proposition \ref{prop:MQw}\iti{i}. This proves \eqref{bw-condb}.
 \end{proof}
 \ppnoi
 To the best of the author's knowledge, the convergence of the steepest-descent variant of the Landweber iteration remains a conjecture under the mere condition that $M_{Q,w}\neq\emptyset$. Under these circumstances, we cannot claim Property \ref{obj} and can therefore only conclude the following partial result.

\begin{corollary}\label{corol2}
Let $Q\in\calB(U,\hat U)$, $w\in\hat U$, $u\up0\in U$ and $(u\upn)\n0$ be generated by the recurrence \eqref{SD-rec}.
\begin{enumerate}[label=(\roman*)]
\setlength{\itemsep}{0.2em}
\item If $M_{Q,w}\cap\ran(Q^*)\neq\emptyset$ and $u\up0\in\ran(Q^*)\oplus\ran(Q^*)^\perp$, $u\upn$  converges to $P_{M_{Q,w}}u\up0$.
\item If $M_{Q,w}=\emptyset$, $\|u\upn\|$ tends to $\infty$.
\end{enumerate}
\end{corollary}

\begin{proof}
We use the fact that \eqref{SD-rec} is an admissible iteration, which allows us to apply the results of Section \ref{sec:admin-family}. Let $U_n$ the $n$th iterate map of \eqref{SD-rec} and $E_n$ be the associated error map.

\iti{i} Assume that $M_{Q,w}\cap\ran(Q^*)\neq\emptyset$ and $u\up0\in\ran(Q^*)\oplus\ran(Q^*)^\perp$. Since $M_{Q,w}\neq\emptyset$, then $w\in\dom(Q^\dagger)$. By Corollary \ref{corol:trans-equiv1}, $E_n(w,u\up0)=E_n(\tilde w,0)$ where
\begin{equation}\label{tw-w}
\tilde w:=w-Qu\up0.
\end{equation}
Defining $u\upn:=U_n(w,u\up0)$ and $\tilde u\upn:=U_n(\tilde w,0)$ for all $n\geq0$, we then have
\begin{equation}\label{en-equ}
u\upn-P_{M_{Q,w}}u\up0=\tilde u\upn-P_{M_{Q,\tilde w}}0,\qquad n\geq0.
\end{equation}
By construction, $(\tilde u\upn)\n0$ satisfies the recurrence \eqref{SD-rec} with $\tilde w$ in place of $w$ and the initial iterate $\tilde u\up0=U_0(\tilde w,0)=0$. To apply Theorem \ref{theo:Kammerer} to $(\tilde u\upn)\n0$, it remains to verify that
$\bar{\tilde w}:=P_{\,\overline{\ran(Q)}}\in\ran(QQ^*)$. It follows from \eqref{tw-w} that
$\bar{\tilde w}= \bar w-Qu\up0$ where $\bar w:=P_{\,\overline{\ran(Q)}}\'w$. By the assumption on $u\up0$, we have $Qu\up0\in Q(\ran(Q^*))=\ran(QQ^*)$ since $Q(\ran(Q^*)^\perp)=Q(\nul(Q))=\{0\}$. But $\bar w\in\ran(QQ^*)$ due to \eqref{bw-condb}. Hence, $\bar{\tilde w}\in\ran(QQ^*)$. Applying Theorem \ref{theo:Kammerer} to $(\tilde u\upn)\n0$, we obtain that the right-hand side of \eqref{en-equ} converges to zero. This proves \iti{i}.

\iti{ii} This follows from Proposition \ref{prop:basic-conv}\iti{ii} given that $r\upn$ converges to zero from Section \ref{subsec:SD-extinct}.
\end{proof}
\ppnoi
When $\ran(Q)$ is closed, $\ran(QQ^*)=\ran(Q)$ (see for example \cite[(1.1.8)]{Tucsnak09}). By Proposition \ref{prop:SD-cond}, $M_{Q,w}\cap\ran(Q^*)\neq\emptyset$ $\Leftrightarrow$ $M_{Q,w}\neq\emptyset$. Meanwhile, $\ran(Q^*)$ is also closed, which makes $\ran(Q^*)\oplus\ran(Q^*)^\perp=U$.
In this case, Corollary \ref{corol2} implies the full version of Property \ref{obj}. Although this provides a reassuring consistency check, its contribution is nevertheless limited. Indeed, the next section establishes linear convergence results under the assumption that $\gamma(Q)>0$, which is equivalent to having $\ran(Q)$ closed from Theorem \ref{theo:Kato}.

\subsection{Linear convergence}

The following linear convergence result is adapted from Theorem 2 from \cite[\S XV.1]{Kantorovich82}.

\begin{theorem}\cite{Kantorovich82}\label{theo:Kantorovich}
Let $A\in\calB(U,U)$ be a boundedly invertible positive self-adjoint operator, $y\in U$, $x^*:=A^{-1}y$ and $(x\upn)\n0$ be generated by the recurrence
\begin{equation}\label{SD-rec0}
x\upnp=x\upn-\alpha_n\'r\upn\quad\text{where}\quad
r\upn:=Ax\upn-y~~\text{and}~~\alpha_n:=\midfrac{\|r\upn\|^2}{\langle Ar\upn,r\upn\rangle},\quad n\geq0.
\end{equation}
Then,
\begin{equation}\label{SD-decay0}
\|x\upn\!-x^*\|\leq\kappa\Big(\midfrac{\kappa^2{-}1}{\kappa^2{+}1}\Big)^{\!n}\'\|x\up0\!-x^*\|,\qquad n\geq0
\end{equation}
where $\kappa:=(\|A\|/\gamma(A))^{1/2}$.
\end{theorem}
\ppnoi
This result is sufficient to establish the linear convergence of iteration \eqref{SD-rec} under the sole assumption that $\gamma(Q)>0$, using properties of admissible iterations in Section \ref{subsub:ex-SD}.

\begin{corollary}\label{corol:SD}
Let $Q\in\calB(U,\hat U)$ such that $\gamma(Q)>0$, $w\in\hat U$ and $(u\upn)\n0$ be generated by the recurrence \eqref{SD-rec}. Then, the error $e\upn:=u\upn-P_{M_{Q,w}}u\up0$ decays in norm as \begin{equation}\label{SD-decay1}
\|e\upn\|\leq\kappa\Big(1-\midfrac{2}{\kappa^2{+}1}\Big)^{\!n}\'\|e\up0\|,\qquad n\geq0
\end{equation}
where $\kappa:=\|Q\|/\gamma(Q)$.
\end{corollary}

\begin{proof}
Let $U_n$ be the $n$th iterate map of \eqref{SD-rec}. By Corollary \ref{corol:trans-equiv2}, there exists $\tilde u\up0\in\nul(Q)^\perp$ such that $e\upn=U_n(0,\tilde u\up0)\in\nul(Q)^\perp$ for all $n\geq0$. Hence, $e\upn$ satisfies the recurrence \eqref{SD-rec} with $w=0$ starting from $e\up0=\tilde u\up0\in\nul(Q)^\perp$. Explicitly,
$$e\upnp=e\upn-\alpha_n\'r\upn\quad\text{where}\quad r\upn:=Q^*Q e\upn\quad\text{and}\quad\alpha_n:=\middfrac{\|r\upn\|^2}{\|Qr\upn\|^2},\quad n\geq0.$$
Then, the sequence $(x\upn)\n0:=(e\upn)\n0$ satisfies the recurrence \eqref{SD-rec0} with $A:=Q^*Q|_{\nul(Q)^\perp}$ and $y=0$. Since $\gamma(A)=\gamma(Q)^2$ \cite{Kulkarni20}, we have $\gamma(A)>0$. Therefore, $A$ satisfies the conditions of Theorem \ref{theo:Kantorovich}, which gives \eqref{SD-decay0}. Since $x^*=A^{-1}y=0$, then $x\upn\!-x^*=e\upn$, which implies that $\|e\upn\|\leq\kappa\,\rho^n\|e\up0\|$ where $\rho:=(\kappa^2{-}1)/(\kappa^2{+}1)=1-2/(\kappa^2{+}1)$.
Finally, $\|A\|=\|Q\|^2$ as well, so that $\kappa:=\|Q\|/\gamma(Q)$.
\end{proof}
\ppnoi
The convergence rate $1{-}2/(\kappa^2{+}1)$ of \eqref{SD-decay1} is to be compared with that of Landweber iteration with variable step sizes equal to $1{-}\eps/\kappa^2$ with $\eps\in(0,1]$ as seen in \eqref{linrate}. The latter yields the best rate of $1{-}1/\kappa^2$. One readily checks that this is larger than $1{-}2/(\kappa^2{+}1)$ exactly when $\kappa>1$. These convergence-rate estimates reflect the well-known advantage of the steepest-descent variant over the basic Landweber iteration with variable step sizes in $[\frac{\eps}{2},1{-}\frac{\eps}{2}]$.

\subsection{Application to alternating projections}\label{subsec:SD-AP}

Applying \eqref{SD-rec} to alternating projections between $U$ and $W$ is straightforward. Applying \eqref{Q} and the identity \eqref{AP-res}, \eqref{SD-rec} becomes
\begin{subequations}\label{SDPOCS}
\begin{align}
r\upn&=u\upn\!-P_U P_W u\upn,\label{SDPOCSa}\\
u\upnp&= u\upn-\alpha_n\'r\upn,\qquad n\geq0\label{SDPOCSb}\\
\text{with}\qquad\qquad\qquad\qquad\qquad\quad&&~~\qquad\qquad\qquad\qquad\qquad\nonumber\\
&\alpha_n=\dfrac{\|r\upn\|^2}{\|P_{V^\perp}r\upn\|^2}.\label{SDPOCSc}
\end{align}
\end{subequations}
This is the steepest-descent variant of the alternating projections between $U$ and $W$.

As in the case of the Landweber iteration, convergence theorems for this iteration follow directly from Corollaries \ref{corol2} and \ref{corol:SD} with the pair $(Q,w)$ of \eqref{Q}. However, as the first corollary only gives a partial result of strong convergence, we restrict ourselves to the second corollary. Using translation arguments similar to those of Section \ref{subsec:lin-alt}, we obtain the following result.

\begin{proposition}\label{prop:SD-AP-lin}
Assume that $\theta_f=\theta_f(U,V)>0$.
Let $(u\upn)\n0$ be defined recursively by the recurrence \eqref{SDPOCS}.
Then, the error sequence $e\upn:=u\upn-P_{M_{U,W}}u\up0$ decays in norm as \begin{equation}\label{SD-AP-lin}
\|e\upn\|\leq\kappa\Big(1-\midfrac{2}{\kappa^2{+}1}\Big)^{\!n}\'\|e\up0\|,\quad n\geq0
\where \kappa:=\sin\theta_p/\sin\theta_f.
\end{equation}
\end{proposition}

\subsection{Comparison with extrapolated alternating projections}

There is an interesting similarity of the steepest-descent variant of alternating projections and the extrapolated alternating projection method (EAPM) \cite{Bauschke06} which was specifically designed for consistent constraints. Applied to our two affine subspaces $U$ and $W$, the latter consists of iterating
\begin{subequations}\label{POCS-rel2}
\begin{eqnarray}
&u\upnp=u\upn\!+\alpha_n\big(P_U P_W u\upn\!-u\upn\big),&\qquad n\geq0.\label{POCS-rel2a}\\
\text{with}\qquad\qquad\qquad~~&&~~\qquad\qquad\qquad\qquad\qquad\nonumber\\
&\alpha_n:=\dfrac{\|P_W u\upn\!-u\upn\|^2}{\|P_U P_W u\upn\!-u\upn\|^2}.\label{POCS-rel2b}
\end{eqnarray}
\end{subequations}
In its most general version, EAPM applies an extra constant coefficient $\rho\in(0,2)$ to $\alpha_n$, but we restrict our attention to $\rho=1$ in our comparison. From \eqref{SDPOCS}, the steepest descent variant can be viewed as iterating \eqref{POCS-rel2a} with the alternative coefficient
\begin{equation}\label{POCS-rel2b'}
\alpha_n:=\dfrac{\|P_U P_W u\upn\!-u\upn\|^2}{\|P_{V^\perp}(P_U P_W u\upn\!-u\upn)\|^2}
\end{equation}
where we used the linearity of $P_U$. An easier way to compare the two methods is to put them into the two-term recurrence
\begin{subequations}\label{EAPM-equiv}
\begin{align}
q\upn&=P_W u\upn\!-u\upn\\
u\upnp&=u\upn+\alpha_n\'P_U q\upn
\end{align}
\end{subequations}
where the coefficients are given by
\begin{equation}\label{rel-EAPM}
\alpha_n=\frac{\|q\upn\|^2}{\|P_U\'q\upn\|^2} \et
\alpha_n=\frac{\|P_U\'q\upn\|^2}{\|P_{V^\perp\!}(P_U\'q\upn)\|^2}.
\end{equation}
in EAPM and the steepest-descent method, respectively.

It may be difficult to explain the different effects of these two versions of $\alpha_n$ by simply inspecting their expressions. Now, the recurrence \eqref{POCS-rel2} turns out to coincide with the method of \cite{Tam21} for two affine subspaces. In this work, $\alpha_n$ is obtained by minimizing $\|P_{U\cap W}u\upnp\!-u\upnp\|$. This contrasts with the minimization of $\|P_Wu\upnp\!-u\upnp\|$ by the steepest-descent method. As an immediate contribution, the latter method allows $W$ to be disjoint from $U$. The consequence of this difference is most clearly illustrated in the simple example where $U=\RR\times\{0\}$ and $W=\{(0,1)\}$ in $\RR^2$. In this case, it can be verified that the iterates of EAPM alternate indefinitely between $(u_0,0)$ and $(-1/u_0,0)$ if $u\up0=(u_0,0)\neq(0,0)$, while the alternating projections and their steepest-descent version converge to $(0,0)$ in a single iteration.

\section{Conjugate gradient method}\label{sec:CG}

We now turn to the conjugate gradient (CG) method introduced in Section \ref{subsub:CG}. For convenience, we recall its recurrence
\begin{subequations}\label{CG-rec}
\begin{align}
u\upnp&= u\upn+\alpha_n\'p\upn,\label{CG-reca}\\
r\upnp&=r\upn+\alpha_n\'Q^*Q\'p\upn,\label{CG-recb}\\
p\upnp&=-r\upnp+\beta_n\'p\upn,\quad\qquad n\geq0\label{CG-recc}\\
\text{where}\qquad\qquad\qquad\qquad\qquad\quad\,\nonumber\\
\alpha_n:=\|r\upn\|^2\big/\|Qp\upn\|^2& \et
\beta_n:=\|r\upnp\|^2\big/\|r\upn\|^2\label{ab2}\\
\text{starting from}\qquad\qquad\qquad\qquad&&\qquad\qquad\qquad\qquad\qquad\nonumber\\
-p\up0=&~r\up0=Q^*(Q u\up0\!-w).\label{CG-recd}
\end{align}
\end{subequations}

\subsection{Optimality criterion}

Theorem 7.3 of \cite{Engl96} shows that the CG method satisfies the following optimality criterion.\footnote
{The correspondence between the notation of Theorem 7.3 of \cite{Engl96} and of the present article is
$(T,y^\delta,x^\delta_k,s_k,p_k)\leftrightarrow(Q,w,u\up{n},r\up{n},p\up{n}).$}

\begin{theorem}\cite{Engl96}
Let $(u\upn)\n0$ satisfy the recurrence \eqref{CG-rec}.  Then,
\begin{eqnarray}
&u\upnp=\displaystyle\argmin_{u\in u\up0+\calK_{n+1}}\Phi_{Q,w}(u),\qquad n\geq0\label{CG-opt}\\
\lefteqn{\hspace{-5mm}\text{where}}&&\qquad\qquad\nonumber\\
&\calK_n:=\Span\big\{r\up0,Ar\up0,\cdots,A^{n-1}r\up0\big\},\quad n\geq1
\quad\text{and}\quad A:=Q^*Q\label{Krylov-def}
\end{eqnarray}
with $\calK_0:=\{0\}$.
\end{theorem}
\ppnoi
For comparison, we recall from \eqref{SD-gen} that the $(n{+}1)$th iterate of steepest descent minimizes $\Phi_{Q,w}(u)$ over the line $\calL_{Q,w}(u\upn)$ defined in \eqref{LQw}.
The next proposition compares the search domains of $\calL_{Q,w}(u\upn)$ and $u\up0\!+\calK_{n+1}$.

\begin{proposition}\label{prop:LQwsub}
Let $(u\upn)\n0$ satisfy the CG recurrence \eqref{CG-rec}. Then,
\begin{equation}\label{LQwsub}
\calL_{Q,w}(u\upn)\subset u\up0\!+\calK_{n+1},\qquad n\geq0.
\end{equation}
\end{proposition}

\begin{proof}
Recall from Section \ref{subsub:CG} that $r\upn$ from the recurrence \eqref{CG-rec} is equal to $Q^*(Qu\upn\!-w)$ for all $n\geq0$ (it satisfies the relation \eqref{generic-a}). On one hand,
\begin{equation}\label{LQwsub'}
\calL_{Q,w}(u\upn)=u\upn+\Span\{r\upn\}\subset u\up0\!+\calK_n+\Span\{r\upn\}.
\end{equation}
On the other hand, $r\upn\!-r\up0=Q^*Q(u\upn\!-u\up0)=A(u\upn\!-u\up0)\subset\calK_{n+1}$ since $u\upn\!-u\up0\in\calK_n$. Thus, $r\upn\in r\up0+\calK_{n+1}=\calK_{n+1}$. Therefore, \eqref{LQwsub'} implies \eqref{LQwsub}.
\end{proof}
\ppnoi
Thus, CG minimizes $\Phi_{Q,w}$ over a domain around the current estimate that is larger than that of steepest descent. This suggests that CG should converge faster.

\subsection{Residual extinction}\label{subsec:CG-res}

Proposition \ref{prop:LQwsub} also leads to a simple proof that the residual of CG intrinsically converges to zero.

It follows from \eqref{CG-opt}, \eqref{LQwsub} and \eqref{dPhi} that
 $$\Phi_{Q,w}(u\upn)-\Phi_{Q,w}(u\upnp)~\geq~ \Phi_{Q,w}(u\upn)-\argmin_{u\in\calL_{Q,w}(u\upn)}\Phi_{Q,w}(u)~\geq~
 \midfrac{\|r\upn\|^2}{2\|Q\|^2}.$$
As in the proof of Corollary \ref{SD-extinct}, we conclude that $\sum\n0\|r\upn\|^2\leq 2\|Q\|^2\' \Phi_{Q,w}(u\up0)<\infty$, which implies that $\lim\limits_{n\to\infty}r\upn=0$.

\subsection{Strong convergence}

The main strong convergence result for the CG method is given by Theorem 2.4 of \cite{Caruso22}. After adapting it to our notation,\footnote
{The correspondence between the notation of Theorem 2.4 from \cite{Caruso22} and of the present article is
$(A,g,f^{[N]},\Re_N)\leftrightarrow(Q^*Q,Q^*w,u\upn,-r\up{n})$. With Proposition \ref{prop:MQw}\iti{iii}, we get the correspondence
$\calS\leftrightarrow U_{Q,w}$. Because $A$ is bounded and convergence is in norm, Theorem 2.4 should be read with $\sigma^*=\sigma=0$ and $C^\infty(A)=\scC_{A,g}(0)=\calH$.}
we obtain the following statement.

\begin{theorem}\cite{Caruso22}\label{theo:Caruso}
Let $Q\in\calB(U,\hat U)$, $w\in\hat U$ and $(u\upn)\n0$ satisfy the recurrence \eqref{CG-rec}. If $M_{Q,w}\neq\emptyset$, then $u\upn$ converges to $P_{M_{Q,w}}u\up0$.
\end{theorem}
\ppnoi
More precisely, \cite{Caruso22} assumes that
\begin{equation}\label{CG-opt2}
u\upn=\displaystyle\argmin_{u\in u\up0+\calK_n}\big\|Q\big(u-P_{M_{Q,w}}u\up0\big)\big\|,\qquad n\geq0\vspace{-2mm}
\end{equation}
where $\calK_n$ is defined in \eqref{Krylov-def} with $r\upn:=Q^*(Qu\upn\!-w)$. This is equivalent to \eqref{CG-opt} for the following reason. First, because $P_{M_{Q,w}}u\up0\in M_{Q,w}$, it follows from Proposition \ref{prop:MQw}\iti{i} that $QP_{M_{Q,w}}u\up0=\bar w:=P_{\,\overline{\ran(Q)}}\'w$, so that $\|Q(u{-}P_{M_{Q,w}}u\up0)\|=\| Qu-\bar w\|$ in \eqref{CG-opt2}. Since $ Qu-\bar w\in\overline{\ran(Q)}$ and $\bar w-w\in\overline{\ran(Q)}^\perp$, then $\|Qu-w\|^2=\| Qu-\bar w\|^2+\|\bar w-w\|^2$. Therefore, \eqref{CG-opt2} is equivalent to \eqref{CG-opt}.

\begin{corollary}\label{corol:Caruso}
For any $Q\in\calB(U,\hat U)$ and $w\in\hat U$,
the iterates generated by the recurrence \eqref{CG-rec} satisfy Property \ref{obj}.
\end{corollary}

\begin{proof}
Property \ref{obj}\iti{i} is already obtained by Theorem \ref{theo:Caruso}. Property \ref{obj}\iti{ii} follows from the residual extinction we proved in Section \ref{subsec:CG-res} and Proposition \ref{prop:basic-conv}\iti{ii}.
\end{proof}

\subsection{Linear convergence}

The following is linear convergence result of the CG method adapted from Theorem 1.2.2 of \cite{Daniel67}.\footnote
{The correspondence between the notation of Theorem 1.2.2 from \cite{Daniel67} and of the present article is
$(M,N,\alpha,k,x,x_n,g_n,p_n,c_n,b_n)\leftrightarrow
(Q,Q^*Q,\kappa^{-2},w,u,u\upn,-r\up{n},p\up{n},\alpha_n,\beta_n)$ with $H=K=I$.}

\begin{theorem}\cite{Daniel67}\label{theo:Daniel}
Assume that $Q\in\calB(U,\hat U)$ is boundedly invertible. Let $(u\upn)\n0$ satisfy the recurrence \eqref{CG-rec}.
Then,
\begin{equation}\label{Daniel-rate}
\|Qu\upn\!-w\|\leq 2\Big(\dfrac{\kappa{-}1}{\kappa{+}1}\Big)^{\!n}\,\|Qu\up0\!-w\|,\qquad n\geq0
\end{equation}
where $\kappa:=\|Q\|/\gamma(Q)$.
\end{theorem}
\ppnoi
Note that \cite{Daniel67} uses a version of CG iteration that employs the coefficients
\begin{equation}\label{ab0}
\alpha_n=\langle r\upn\!,p\upn\rangle\big/\|Qp\upn\|^2 \et
\beta_n=\big\langle r\upnp\!,Q^*Qp\upn\big\rangle\big/\|Qp\upn\|^2,\qquad n\geq0.\tag{\ref{ab2}'}
\end{equation}
It is however known in the CG method literature (see for example \cite[\S5.1]{Nocedal06}) that \eqref{ab2} and \eqref{ab0} generate the same sequence of iterates $(u\upn,r\upn,p\upn)$ for a given $u\up0$.

As in Corollary \ref{corol:SD} for the steepest descent method,  we deduce from Theorem \ref{theo:Daniel} the linear convergence of iteration \eqref{CG-rec} under the sole assumption that $\gamma(Q)>0$.

\begin{corollary}\label{corol:CG}
Let $Q\in\calB(U,\hat U)$ such that $\gamma(Q)>0$, $w\in\hat U$ and $(u\upn)\n0$ satisfy the recurrence \eqref{CG-rec}. Then, the error $e\upn:=u\upn-P_{M_{Q,w}}u\up0$ decays in norm as \begin{equation}\label{CG-decay1}
\|e\upn\|\leq 2\kappa\Big(1-\midfrac{2}{\kappa{+}1}\Big)^{\!n}\'\|e\up0\|,\qquad n\geq0
\end{equation}
where $\kappa:=\|Q\|/\gamma(Q)$.
\end{corollary}

\begin{proof}
We proceed as in the proof of Corollary \ref{corol:SD}.
Let $U_n$ be the $n$th iterate map of \eqref{CG-rec}. By Corollary \ref{corol:trans-equiv2}, there exists $\tilde u\up0\in\nul(Q)^\perp$ such that  $e\upn=U_n(0,\tilde u\up0)\in\nul(Q)^\perp$ for all $n\geq0$. Hence, $e\upn$ satisfies the recurrence \eqref{CG-rec} with $w=0$. Thus, Theorem \ref{theo:Daniel} can be applied on the sequence $(e\upn)\n0$ with $Q'$ in place of $Q$, where $Q'$ is the restriction of $Q$ to the domain $\nul(Q)^\perp$ and the codomain $\ran(Q)$. Because $\gamma(Q)>0$, $Q'$ is boundedly invertible. Since $w=0$ in the iteration, it follows from \eqref{Daniel-rate} that $\|Qe\upn\|\leq 2\rho^n\'\|Qe\up0\|$ for all $n\geq0$, where $\rho:=(\kappa{-}1)/(\kappa{+}1)=1-2/(\kappa{+}1)$. But since $\gamma(Q)\|e\upn\|\leq\|Qe\upn\|$ and $\|Qe\up0\|\leq\|Q\|\'\|e\up0\|$, we then obtain $\|e\upn\|\leq 2\kappa\rho^n\|e\up0\|$ for all $n\geq0$.
\end{proof}
\ppnoi
The rate of convergence $1-2/(\kappa+1)$ is strictly smaller than the rate $1-2/(\kappa^2+1)$ obtained in \eqref{SD-decay1} for the steepest descent method. This analytically reflects the well-known fact that the CG method is faster than gradient descent methods. Such an analytical comparison of convergence rates was first established in \cite{Daniel67}.

\subsection{Application to alternating projections}\label{subsec:CG-AP}

To obtain the CG version of alternating projections, one can simply substitute the pair $(Q,w)$ from \eqref{Q} into \eqref{CG-rec}. However, to obtain a recurrence closer to the steepest descent version of the alternating projections in \eqref{SDPOCS}, it is more convenient to start from the following equivalent version of the CG recurrence
\begin{subequations}\label{sys1}
\begin{align}
r\upn&=Q^*(Qu\upn\!-w\upn),\label{sys1a}\\
p\upn&=-r\upn+\beta_{n-1}\'p\upnm,\label{sys1b}\\
u\upnp&= u\upn+\alpha_n\'p\upn,\quad\qquad n\geq0\label{sys1c}
\end{align}
\end{subequations}
with the convention $\beta_{-1}:=0$. We simply replaced the residual recurrence \eqref{CG-recb} by the equivalent identity \eqref{sys1a}, established in Section 3.2.3, and replaced \eqref{CG-recc} by its index-shifted form \eqref{sys1b}. By applying \eqref{Q} and the identity \eqref{AP-res}, \eqref{sys1} becomes
\begin{subequations}\label{CGPOCS}
\begin{align}
r\upn&=u\upn\!-P_U P_W u\upn,\label{CGPOCSa}\\
p\upn&=-r\upn+\beta_{n-1}\'p\upnm,\label{CGPOCSb}\\
u\upnp&= u\upn+\alpha_n\'p\upn,\quad\qquad n\geq0\label{CGPOCSc}\\
\hspace{-5mm}\text{with}\qquad\qquad\qquad\qquad\qquad~\nonumber\\
\hspace{-5mm}\alpha_n=\|r\upn\|^2/\|P_{V^\perp}p\upn\|^2\quad&\quad\text{and}\qquad
\beta_{n-1}=\|r\upn\|^2/\|r\upnm\|^2,\qquad n\geq0\label{abPOCS}
\end{align}
\end{subequations}
and $\beta_{-1}=0$. The following theorem follows directly from Corollary \ref{corol:Caruso} and \ref{corol:CG} with the pair $(Q,w)$ defined in \eqref{Q}.

\begin{proposition}\label{prop:CGPOCS}
Let $(u\upn)\n0$ satisfy the recurrence \eqref{CGPOCS}.
\begin{enumerate}[label=(\roman*)]
\setlength{\itemsep}{0.2em}
\item $(u\upn)\n0$ satisfies Property \ref{prop:Bauschke}.
\item If $\theta_\f=\theta_\f(U,V)>0$, then $M_{U,W}\neq\emptyset$ and  the error $e\upn:= u\upn\!-P_{M_{U,W}}u\up0$ decays in norm as
\begin{equation}\label{CGPOCS-decay}
\|e\upn\|\leq 2\kappa\Big(1-\midfrac{2}{\kappa{+}1}\Big)^{\!n}\,\|e\up0\|,\quad n\geq0\where  \kappa:=\sin\theta_p/\sin\theta_\f.
\end{equation}
\end{enumerate}
\end{proposition}

\section{Conclusion}

In this paper, we reformulated alternating projections between two affine subspaces of a Hilbert space as the minimization of an associated least-squares functional. This optimization viewpoint identifies the classical alternating projection method as the unit-step Landweber iteration and provides a common framework in which optimization algorithms can be systematically translated into alternating projection algorithms.

To formalize this connection, we introduced the class of admissible residual-state iterations, which encompasses the Landweber, steepest-descent, and conjugate-gradient methods. The notions of residual extinction and translation equivariance were shown to provide a unified mechanism for transferring convergence results from least-squares optimization to alternating projections. This framework yields strong convergence results in both the consistent and inconsistent settings, while separating the analysis of optimization algorithms from their geometric interpretation in terms of projections.

As applications, we derived accelerated variants of alternating projections based on steepest descent and conjugate gradients. Under closed-range assumptions, linear convergence was established and the convergence rates were expressed in terms of the Friedrichs angle and the largest principal angle between the underlying subspaces. The resulting comparison highlights the expected acceleration provided by the conjugate-gradient method over classical alternating projections.

Beyond the algorithms studied here, the proposed optimization framework suggests a general methodology for constructing and analyzing alternating projection methods through their least-squares counterparts. Future work could investigate additional optimization techniques within this framework, including accelerated first-order methods, quasi-Newton methods, and regularization strategies for ill-posed problems.

\section*{Acknowledgement}

We would like to thank Andrzej Cegielski for helpful discussions on the background of alternating projections.

\bibliographystyle{tfs}
\bibliography{../reference}{}

\appendix

\section{Proof of Proposition \ref{prop:admissible-CG}}\label{app:admissible-CG}

Let $u\upn$, $r\upn$ and $p\upn$ be recursively defined by \eqref{sys2} for $n\geq0$. We show that the recurrence \eqref{generic}holds with the mapping $J_n$ constructed in \eqref{CG-int} and the states
\begin{equation}\label{sn}
d\upn:=\alpha_n\'p\upn,\quad n\geq0 \et s\upn:=(r\upnm,p\upnm),\quad n\geq1.
\end{equation}
With the above definition of $d\upn$, \eqref{generic-c} directly results from \eqref{sys2a} for all $n\geq0$.
Since \eqref{generic-a} already holds for $n=0$ from \eqref{sys2d}, assume it holds for some $n\geq0$. By \eqref{sys2a} and \eqref{sys2b} that $Q^*(Qu\upnp\!-w)=Q^*(Q(u\upn\!+\alpha_n\'p\upn)-w)=r\upn+\alpha_n Q^*Q\,p\upn=r\upnp$. Thus, \eqref{generic-a} holds for all $n\geq0$.

We now turn to \eqref{generic-b}.
It follows from (\ref{sys2}b)-(\ref{sys2}d) and \eqref{CG-int-RP} that
\begin{subequations}\label{rp}
\begin{align}
r\upnp&=r\upn+(\|r\upn\|^2/\|Qp\upn\|^2)\'Q^*Qp\upn=R(r\upn\!,p\upn)=R(s\upnp),
\label{rpa}\\
p\upnp&=-r\upnp+(\|r\upnp\|^2/\|r\upn\|^2)\'p\upn=P\big(r\upnp,(r\upn\!,p\upn)\big)\nonumber\\
&=P(r\upnp\!,s\upnp),\qquad n\geq0.\label{rpb}
\end{align}
\end{subequations}
Thus,
$$s\upnp=(r\upn\!,p\upn)=\big(R(s\upn),P(r\upn\!,s\upn)\big)=S(r\upn,s\upn),\qquad n\geq1$$
with the function $S$ in \eqref{CG-int-DS}. Let $n\geq1$.
Since $d\upn=\alpha_n\'p\upn=\big(\|r\upn\|^2/\|Qp\upn\|^2\big)p\upn$ from \eqref{ab} and
$p\upn=P(r\upn\!,s\upn)$ from \eqref{rpb}, then
$d\upn=D(r\upn\!,s\upn)$ with the function $D$ in \eqref{CG-int-DS} . Thus, $(d\upn,s\upnp)=J_n(r\upn,s\upn)$ with the function $J_n$ in \eqref{CG-int-Jn}. Then, \eqref{generic-b} is satisfied for all $n\geq1$. It remains to verify this for $n=0$. Since $\alpha_0=\|r\up0\|^2/\|Qr\up0\|^2$ from \eqref{ab} and $p\up0=-r\up0$ from \eqref{sys2d}, it follows from \eqref{CG-int-J0} that
$$J_0(r\up0\!,0)=\big(-\alpha_0\'r\up0,(r\up0\!,-r\up0)\big)
=\big(\alpha_0\'p\up0,(r\up0\!,p\up0)\big)=\big(d\up0,s\up{1}\big).$$
This proves \eqref{generic-b} for $n=0$. This completes the proof.

\section{Spectral radius of restriction of an operator}

\begin{lemma}\label{lem:restr-rad}
Let $R$ be a bounded linear operator on $H$ and $S$ be closed subspace that contains $\ran(R)$. Then, $\rho(R)=\rho(R|_S)$.
\end{lemma}

\begin{proof}
Note first that we regard $R|_S$ as an endomorphism of $S$  so that $\rho(R|_S)$ is well defined.
It is well known (see for example Theorem 10.13 of \cite{Rudin91}) that, for any bounded linear operator $T$ on a given Hilbert space,
\begin{equation}\label{rho-charac}
\rho(T)=\lim_{n\to\infty}\rho_n(T)\where \rho_n(T):=\|T^n\|^{1/n}.
\end{equation}
Since $\ran(R)\subset S$, an induction argument shows that
$$R^n\'u=(R|_S)^{n-1}Ru,\qquad u\in H,~n\geq1.$$
Hence, $R^n|_S=(R|_S)^{n-1}R|_S=(R|_S)^n$ and
$$\|(R|_{ S})^n\|=\|R^n|_{ S}\|\leq \|R^n\|\leq\big\|(R|_S)^{n-1}\big\|\,\|R\|,\qquad n\geq1.$$
Thus, $\rho_n(R|_S)\leq \rho_n(R)\leq\rho_{n-1}(R|_S)^{(n-1)/n}\,\|R\|^{1/n}$ for all $n\geq1$. Taking the limit as $n\to\infty$, we obtain $\rho(R|_S)\leq\rho(R)\leq \rho(R|_S)$.
\end{proof}

\page
\tableofcontents

\end{document}